\documentclass[10pt]{article}
\usepackage{epsfig}
\usepackage{amsmath}
\usepackage{amssymb}
\usepackage{graphicx}
\usepackage[latin1]{inputenc}
\usepackage{amsthm}
\usepackage[french,english]{babel}
\usepackage[active]{srcltx}
\usepackage{bbm} 
\usepackage{color}
\usepackage{enumerate}
\usepackage{latexsym}
\usepackage{graphicx,epsfig}
\def\BibTeX{{\rm B\kern-.05em{\sc i\kern-.025em b}\kern-.08emalors arret
    T\kern-.1667em\lower.7ex\hbox{E}\kern-.125emX}}

\newcommand{\N}{\mathbb{N}}
\newcommand{\Z}{\mathbb{Z}}
\newcommand{\R}{\mathbb{R}}
\newcommand{\E}{\mathbb{E}}
\newcommand{\I}{\mathbb{I}}

\newcommand{\Var}{\mbox{Var}\,}
\newcommand{\Cov}{\mbox{Cov}\,}

\newcommand{\Loi}{\mathcal{L}}

\newcommand{\tend}{\overset{d}{\underset{N\rightarrow\infty}{\longrightarrow}}}
\newcommand{\ba}{\begin{eqnarray}}
\newcommand{\ea}{\end{eqnarray}}
\newcommand{\ban}{\begin{eqnarray*}}
\newcommand{\ean}{\end{eqnarray*}}
\newcommand{\be}{\begin{equation}}
\newcommand{\ee}{\end{equation}}

\def\limiteloiN{\renewcommand{\arraystretch}{0.5}
\begin{array}[t]{c}
\stackrel{{\Loi}}{\longrightarrow} \\
{\scriptstyle N\rightarrow\infty}
\end{array}\renewcommand{\arraystretch}{1}}

\def\limiteprobaN{\renewcommand{\arraystretch}{0.5}
\begin{array}[t]{c}
\stackrel{{\cal P}}{\longrightarrow} \\
{\scriptstyle N\rightarrow\infty}
\end{array}\renewcommand{\arraystretch}{1}}

\def\limitN{\renewcommand{\arraystretch}{0.5}
\begin{array}[t]{c}
\stackrel{}{\longrightarrow} \\
{\scriptstyle N\rightarrow\infty}
\end{array}\renewcommand{\arraystretch}{1}}

\def\limiteN{\renewcommand{\arraystretch}{0.5}
\begin{array}[t]{c}
\stackrel{}{\longrightarrow} \\
{\scriptstyle N\rightarrow\infty}
\end{array}\renewcommand{\arraystretch}{1}}

\newtheorem{thm}{Theorem}
\newtheorem{rem}{Remark}

\newtheorem{lem}{Lemma}
\newtheorem{prop}{Proposition}

\newtheorem{popy}{Property}

\date{}

\begin{document}



\title{Adaptive semiparametric wavelet estimator and goodness-of-fit test for long memory linear processes}
\author{\centerline{Jean-Marc Bardet$^a$ and Hatem Bibi$^a$} \\
\small {\tt bardet@univ-paris1.fr}, \small {\tt
hatem.bibi@malix.univ-paris1.fr}\\
~\\
{\small $^a$ SAMM, Universit\'e Paris 1, 
90 rue de Tolbiac, 75013 Paris, FRANCE.}
}
\maketitle
~\\
 \begin{abstract}
This paper is first devoted to study an adaptive wavelet based estimator of the long memory parameter for linear processes in a general semi-parametric frame. This is an extension of Bardet {\it et al.} (2008) which only concerned Gaussian processes. Moreover, the definition of the long memory parameter estimator is modified and asymptotic results are improved even in the Gaussian case. Finally an adaptive goodness-of-fit test is also built and easy to be employed: it is a chi-square type test. Simulations confirm the interesting properties of consistency and robustness of the adaptive estimator and test.  
\end{abstract}
 \section{Introduction}
The long-memory processes are now a subject area well studied and often applied (see for instance the book edited by Doukhan {\it et al}, 2003). The most famous long-memory stationary time series are the fractional Gaussian noises (fGn) with Hurst parameter $H$ and FARIMA$(p,d,q)$ processes. For both these time series, the spectral density $f$ in $0$  follows a power law: $f(\lambda) \sim C\, \lambda^{-2d}$ where $H=d+1/2$ in the case of the fGn. This behavior of the spectral density is generally a definition adopted for a stationary long memory (or long range dependent) process even if this definition requires the existence of a second order moment. \\
There are a lot of statistical results relative to the estimation of the long memory parameter $d$. First and main results in this direction were obtained for parametric models with the essential papers of Fox and Taqqu (1986) and Dahlhaus (1989) for Gaussian time series, Giraitis and Surgailis (1990) for linear processes and Giraitis and Taqqu (1999) for non linear functions of Gaussian processes. \\
However and especially for numerical applications, parametric estimators are not really robust and can induce no consistent estimations. Thus, the research  is now rather focused on semiparametric estimators of the $d$. Different approaches were considered: the famous and seminal R/S statistic (see Hurst, 1951), the log-periodogram estimator (see Moulines and Soulier, 2003), the local Whittle estimator (see Robinson, 1995) or the wavelet based estimator (see Veitch {\it et al}, 2003, Moulines {\it et al}, 2007 or Bardet {\it et al}, 2008). All these estimators require the choice of an auxiliary parameter (frequency bandwidth, scales,...) but adaptive versions of original estimators are generally built for avoiding this choice. In a general semiparametric frame, Giraitis {\it et al} (1997) obtained the asymptotic lower bound for the minimax risk of estimating $d$, expressed as a function of the second order parameter of the spectral density expansion around $0$. Thus, several adaptive semiparametric are proved to follow an oracle property up to multiplicative logarithm term. But simulations (see for instance Bardet {\it et al}, 2003 or 2008) show that the most accurate estimators are local Whittle, global log-periodogram or wavelet based estimators. \\
~\\
The use of a wavelet based estimator for estimating $d$ was first proposed in Abry {\it et al.} (1998) after preliminary studies devoted to selfsimilar processes. Then Bardet {\it et al.} (2000) provided proofs of the consistency of such an estimator in a Gaussian semiparametric frame. Moulines {\it et al.} (2007) improved these results, proved a central limit theorem for the estimator of $d$ and showed that this estimator is rate optimal for the minimax criterion. Finally, Roueff and Taqqu (2009a) established similar results in a semiparametric frame for linear processes. All these papers were obtained using a wavelet analysis based on a discrete multi-resolution wavelet transform, which notably allows to compute the wavelet coefficients with the fast Mallat's algorithm. However, there remains a gap in these papers: in a semiparametric frame the ``optimal'' scale used for the wavelet analysis is depending on the second order expansion of the spectral density around $0$ frequency and these papers consider that the power of the second order expansion is known while this is unknown in practice. Two papers proposed a method for automatically selecting this ``optimal'' scale in the Gaussian semiparametric frame. Firstly, Veitch {\it et al.} (2003) using a kind of Chi-square test which provides convincing numerical results but the consistency of this procedure is not established. Secondly, Bardet {\it et al.} (2008) proved the consistency of a procedure for choosing optimal scales based on the detection of the ``most linear part'' of the log-variogram graph. In this latter article, the ``mother'' wavelet is not necessary associated to a multi-resolution analysis: the time consuming is clearly more important but a large choice of continuous wavelet transforms can be chosen and the choice of scales is not restricted to be a power of $2$. \\
~\\
The present article is devoted to an extension of the article Bardet {\it et al.} (2008). Three main improvements are obtained: 
\begin{enumerate}
\item The semiparametric Gaussian framework of Bardet {\it et al.} (2008) is extended to a semiparametric framework for linear processes. The same automatic procedure of the selection of the optimal scale is also applied and this leads to adaptive estimators. 
\item As in Bardet {\it et al.} (2008), the ``mother'' wavelet is not restricted to be associated to a discrete multi-resolution transform. Moreover we modified a little the definition of the sample variance of wavelet coefficients (variogram). The result of both these positions is a multidimensional central limit theorem satisfied by the logarithms of variograms with an extremely simple asymptotic covariance matrix (see (\ref{cov})) only depending on $d$ and the Fourier transform of the wavelet function. Hence it is easy to compute an adaptive pseudo-generalized least square estimator (PGLSE) of $d$ which is proved to satisfy a CLT with with an asymptotic variance smaller than the one of the adaptive (respectively non-adaptive) ordinary least square estimator of $d$ respectively considered  in Bardet {\it et al.} (2008) and Roueff and Taqqu (2009). Simulations confirm confirm the good performance of this PGLSE. 
\item Finally, an adaptive goodness-of-fit test can be built from this PGLSE. It consists in a normalized sum of the squared PGLS-distance between between the PGLS-regression line and the points. We prove that this test statistic converges in distribution to a chi-square distribution. Thus it is a very simple test to be computed since the asymptotic covariance matrix is easy to be approximated. When $d>0$ this test is a long memory test. Moreover, simulations show that this test provides good properties of consistency under $H_0$ and reasonable properties of robustness under $H_1$.  
\end{enumerate}
For all these reasons, we can say that this paper is an achievement of the article Bardet {\it et al.} (2008). Moreover, the adaptive PGLS estimator and test represent an interesting extension of the paper Roueff and Taqqu (2009). \\
~\\
We organized the paper as follows. Section \ref{CLT} contains the assumptions, definitions and a first multidimensional central limit theorem, while Section \ref{Adapt} is devoted to the construction and consistency of the adaptive PGLS estimator and goodness-of-fit test. In Section \ref{Simu} we illustrate with Monte-Carlo simulations the convergence of the adaptive estimator and we compare these results to those obtained with other efficient semiparametric estimators; moreover we study the consistency and robustness properties of the adaptive goodness-of-fit test. The proofs are provided in Section \ref{Proofs}.
\section{A central limit theorem for the sample variance of wavelet coefficients}\label{CLT}
For $d<1/2$ and $d'>0$, this paper deals with  the following
semi-parametric framework:\\
~\\
{\bf Assumption A$(d,d')$:} {\it  $X=(X_{t})_{t\in \mathbb{Z}}$ is a zero mean stationary linear process, {\it i.e.}
$$
X_{t}=\sum_{s\in \mathbb{Z}}\alpha(t-s)\xi_{s},\quad t\in \mathbb{Z},\quad \mbox{where}
$$         
\begin{itemize}
\item  $(\xi_{s})_{s\in \Z}$ is a sequence of independent identically distributed random variables such that the distribution of $\xi_0$ is symmetric, {\it i.e.} $\Pr (\xi_0>M)=\Pr (\xi_0<-M)$ for any $M\in \R$, and $\E |\xi_0|^4 <\infty$ with $\E \xi_0=0$, $\Var \xi_0=1$ and $\mu_4:=\E \xi_0^4$; 
\item  $(\alpha(t))_{t\in \Z}$ is a sequence of real numbers such that there exist $c_d>0$ and $c_{d'} \in \R$ satisfying 
\begin{equation}\label{alpha}
|\widehat{\alpha}(\lambda)|^2=\frac 1 {\lambda^{2d}} \big (c_d+ c_{d'} \lambda^{d'}(1+\varepsilon(\lambda)) \big ) \quad \mbox{for any}\quad 
\lambda \in [-\pi, \pi],  
\end{equation}
where $\widehat \alpha (\lambda):= \frac 1 {2\pi} \, \sum_{k \in \Z} \alpha(k) e^{-ik \lambda}$, with $\varepsilon(\lambda)\to 0$ ($\lambda \to 0$).          
\end{itemize} }
~\\
As a consequence, the spectral density $f$ of $X$ is such that 
\begin{eqnarray}\label{dens}
f(\lambda)=2\pi\,|\widehat{\alpha}(\lambda)|^{2}
\end{eqnarray}
and satisfies the same kind of expansion than (\ref{alpha}). Thus, if $d\in (0,1/2)$ the process $X$ is a long-memory process, and if $d\leq 0$ a short memory process (see Doukhan {\it
et al.}, 2003).\\
~\\
Now define $\psi:\R \to \R$ the wavelet function. Let $k\in \N^*$. We consider the following assumption on $\psi$:\\
~\\
{\bf Assumption $\Psi(k)$:} {\it $\psi:\R \to \R$ is such that
\begin{enumerate}
\item the support of $\psi$ is included in $(0,1)$;
\item $\displaystyle \int_0^1  \psi(t)\,dt=0$;
\item $\psi \in {\cal C}^k(\R)$.
\end{enumerate}}
\noindent Straightforward implications of these assumptions are $\psi^{(j)}(0)=\psi^{(j)}(1)=0$ for any $0\leq j \leq k$ and $\widehat \psi (u) \sim C \, u^k$ $(u \to 0)$ with $C$ a real number not depending on $u$. \\
Define also the Fourier transform of $\psi$, {\it i.e.} $\widehat \psi (u):=\int_0^1
\psi(t)\,e^{-iut}dt$. Assumption Assumption $\Psi(k)$ also implies that $\widehat{\psi}$ has a fast decay at infinity.\\
~\\
If $Y=(Y_t)_{t\in \R}$ is a continuous-time process, for $(a,b)\in
\R_+^*\times \R$, the "classical" wavelet coefficient $d(a,b)$ of
the process $Y$ for the scale $a$ and the shift $b$ is
$d(a,b) :=\frac 1{\sqrt a} \int_{\R}  \psi(\frac{t-b}{a})Y_t \,dt.$
However, a process $X$ satisfying Assumption A$(d,d')$ is a discrete-time process, and from a path $(X_1,\ldots,X_N)$ we  define the
wavelet coefficients of $X$ by 
\begin{eqnarray}\label{coeff_e}
e_N(a,b) := \frac{1}{\sqrt{a}}\sum_{t=1}^{N}X_{t}\psi(\frac{t-b}{a}) =\frac{1}{\sqrt{a}}\sum_{t=1}^{N}\sum_{s\in\mathbb{Z}}\alpha(t-s)\psi(\frac{t-b}{a})\xi_{s}
\end{eqnarray}
for $(a,b)\in \N_+^*\times \Z$. Then,
\begin{popy}\label{cor1}
Under Assumption A$(d,d')$ with $d<1/2$ and $d'>0$, and if $\psi$ satisfies Assumption $\Psi(k)$ with $k>d'$,  
then $(e(a,k))_{b\in \{1,\ldots,N-a\}}$ is a zero mean  stationary linear process and
\begin{eqnarray}\label{equiD'}
\E (e^{2}(a,0)) = 2\pi \, c_d \, \Big ( K_{(\psi,2d)} \, a^{2d}+\frac {c_{d'}}{c_d}
K_{(\psi,2d-d')} \, a^{2d-d'}\Big )+ o\big (a^{2d-d'}\big
)\quad \mbox{when} \quad a\to \infty,
\end{eqnarray}
with $K_{(\psi,D)}$ such that
\begin{eqnarray}\label{Kpsi}
K_{(\psi,\alpha)}:=\int_{-\infty}^\infty |\widehat \psi (u)|^2 \,
|u|^{-\alpha}du>0\quad \mbox{for all $\alpha <1$}.
\end{eqnarray} 
\end{popy}
\noindent The proof of this property, like all the other proofs, is provided in Section \ref{Proofs}. \\
Let  $(X_1,\ldots,X_N)$ be a sampled path of $X$ satisfying Assumption A$(d,d')$.  Property \label{vard} allows an estimation of $2d$
from a log-log regression, as soon as a consistent estimator of $\E
(e^2(a,0))$ is provided. For this and with $1 \leq a<N$, consider the sample variance of the wavelet coefficients, 
\begin{eqnarray} \label{samplevar}
 T_{N}(a): = \frac{1}{N-a}\sum
_{k=1}^{N-a}e^{2}(a,k).
\end{eqnarray}
\begin{rem}
In Bardet {\it et al.} (2000), (2008) or in Moulines {\it et al.} (2007) or Roueff and Taqqu (2009), the considered sample variance of wavelet coefficients is 
\begin{eqnarray} \label{samplevarusual} 
V_N(a):=\frac{1}{[N/a]}\sum
_{k=1}^{[N/a]}e^{2}(a,ak)
\end{eqnarray} 
(with $a=2^j$ in case of multiresolution analysis). The definition (\ref{samplevar}) has a drawback and two advantages with respect to this usual definition (\ref{samplevarusual}): it is not adapted to the fast Mallat's algorithm and therefore more time consuming, but it leads to more a simple expression of the asymptotic variance and simulations exhibit that this asymptotic variance is smaller that the one obtained with (\ref{samplevarusual}).
\end{rem}
The following proposition specifies a central limit theorem
satisfied by $\log \widetilde{T}_{N}(a)$, which provides the first step
for obtaining the asymptotic properties of the estimator by log-log
regression. More generally, the following multidimensional central
limit theorem for a vector $(\log \widetilde{T}_{N}(a_i))_i$ can be
established,
\begin{prop}\label{tlclog}
Under Assumption A$(d,d')$, $d<1/2$ and $d'>0$, and if $\psi$ satisfies Assumption $\Psi(k)$ with $k\geq 2$. 
Define $\ell \in \N \setminus \{0,1\}$ and $(r_1,\cdots,r_\ell) \in
(\N^*)^\ell$. Let $(a_n)_{n \in \N}$ be such that $N/a_N \limitN
\infty$ and $a_N \, N^{-1/(1+2d')}\limitN \infty$. Then,
\begin{equation}\label{CLTSN}
\sqrt{\frac{N}{a_N}} \Big(\log {T}_{N}(r_ia_N)-2d \log (r_ia_N)-\log\big (\frac {c_d}{2\pi}K_{(\psi,2d)} \big )\Big)_{1\leq
i\leq \ell}\tend\mathcal{N}_{\ell}\big (0\, ;\,
\Gamma(r_1,\cdots,r_\ell,\psi,d)\big ),
\end{equation}
with $\Gamma(r_1,\cdots,r_\ell,\psi,d)=(\gamma_{ij})_{1\leq i,j\leq
\ell}$ the covariance matrix such that
\begin{eqnarray}\label{cov}
\gamma_{ij}&=&4\pi \, \frac {(r_ir_j')^{1-2d}}{K^2_{(\psi,2d)} }\int_{-\infty}^{\infty} \hspace{-2mm}   \frac{ \big |\widehat \psi (r_i\lambda) \big |^2  |\widehat \psi (r_j\lambda) \big |^2}{\lambda^{4d}} d\lambda.
\end{eqnarray}
\end{prop}
\section{An adaptive estimator of the memory parameter and an adaptive goodness-of-fit test}\label{Adapt}
The CLT of Proposition \ref{tlclog} is very interesting because it has several consequences. We will see that the (quite) simple expression of the asymptotic covariance matrix is an important advantage  compared to the complicated expression of the asymptotic covariance obtained in the case of a multiresolution analysis (see Roueff and Taqqu, 2009a). First it allows to obtain an estimator $\widehat  d_N$ of $d$ by using an ordinary least square estimation. Hence, define
\begin{equation}\label{defd1} 
 \widehat  d_N(a_N)   :=\big ( 0 ~\frac 1 2 \big ) \,(Z_{a_N}' \, Z_{a_N})^{-1} Z_{a_N}' \big(\log {T}_{N}(r_ia_N)\big)_{1\leq i\leq \ell}\quad\mbox{with}\quad Z_{a_N}=\left ( \begin{array}{cc} 1 & \log (a_N) \\ 1 & \log (2a_N) \\ \vdots & \vdots \\ 1 & \log (\ell a_N) \end{array} \right).
\end{equation}
\begin{rem}
From Proposition \ref{tlclog}, it is not possible to chose $(r_1,\ldots,r_\ell)$ for minimizing the asymptotic covariance matrix $\Gamma(r_1,\cdots,r_\ell,\psi,d)$ without knowing the value of $d$. Hence, in the sequel we will only consider the choice $(r_1,r_2,\cdots,r_\ell)=(1,2,\ldots,\ell)$.
\end{rem}
\noindent Then, it is clear from Proposition \ref{tlclog} that $\widehat  d_N(a_N)$ converges to $d$ following a central limit theorem with convergence rate $\sqrt{\frac{N}{a_N}}$ when $a_N$ satisfies the condition $a_N \, N^{-1/(1+2d')}\limitN \infty$. \\
However, in practice, $d'$ is unknown. In Bardet {\it et al.} (2008), an automatic procedure for choosing an ``optimal'' scale $a_N$ has been proposed. We are going to apply again this procedure after recalling its principle: for $\alpha \in (0,1)$, define 
$$
Q_N(\alpha,c,d)= \Big (Y_N(\alpha)-Z_{N^\alpha} \, \big(
\begin{array}{c} c \\ 2d \end{array}\big ) \Big )' \cdot
\Big (Y_N(\alpha)-Z_{N^\alpha} \, \big( \begin{array}{c} c \\
2d \end{array}\big ) \Big ), \quad \mbox{with}\quad Y_N(\alpha) =\big (\log {T}_{N}(iN^\alpha)\big )_{1\leq
i\leq \ell} .
$$
$Q_N(\alpha,c,d)$ corresponds to a squared distance between the
$\ell$ points $\big (\log (i  N^\alpha)\,,\,\log T_N(i 
N^\alpha) \big )_i$ and a line. It can be minimized first by defining for $\alpha \in (0,1)$
$$
\widehat Q_N(\alpha)=Q_N(\alpha,\widehat c(N^\alpha), 2\widehat d(N^\alpha))\quad\mbox{with}\quad \Big(
\begin{array}{c} \widehat c(N^\alpha)\\ 2\widehat d(N^\alpha) \end{array}\Big )=\big (Z_{N^\alpha}'Z_{N^\alpha}\big )^{-1} Z_{N^\alpha}'Y_N(\alpha) ;
$$
and then define $\widehat \alpha_N$ by:
\begin{eqnarray*}
\widehat Q_N(\widehat \alpha_N )=\min_{\alpha \in {\cal A}_N}
\widehat Q_N(\alpha)\quad \mbox{where}\quad {\cal A}_N=\Big \{\frac {2}{\log
N}\,,\,\frac { 3}{\log N}\,, \ldots,\frac {\log [N/\ell]}{\log N}
\Big \}.
\end{eqnarray*}
\begin{rem}
As it was also claimed in Bardet {\it et al.} (2008), in the definition of the set $ {\cal A}_N$, $\log N$ can be replaced by any sequence negligible with respect to any power law of $N$. Hence, in numerical applications we will use $10\, \log N$ which significantly increases the precision of $\widehat \alpha_N$. 
\end{rem}
\noindent Under the assumptions of Proposition \ref{tlclog}, one obtains (see the proof in Bardet {\it et al.}, 2008),
$$
\widehat \alpha_N =\frac {\log {\widehat a_N}}{\log N}
\limiteprobaN \alpha^*=\frac 1 {1+2d'}.
$$
Then define:
\begin{eqnarray}\label{d1}
\widehat {\widehat d_N}:=\widehat d(N^{\widehat \alpha_N})\quad \mbox{and} \quad \widehat \Gamma_N:=\Gamma(1,\cdots,\ell,\widehat {\widehat d_N},\psi).
\end{eqnarray}
It is clear that  $\widehat {\widehat d_N} \limiteprobaN d$ (a convergence rate can also be found in Bardet {\it et al.}, 2008) and therefore, from the expression of $\Gamma$ in (\ref{cov}) and its smoothness with respect to the variable $d$, $\widehat \Gamma_N \limiteprobaN \Gamma(1,\cdots,\ell,d,\psi)$. Thus it is possible to define a (pseudo)-generalized least square estimator (PGLSE) of $d$. Before this, define 
$$
\widetilde \alpha_N:=\widehat \alpha_N+ \frac {6 \widehat
\alpha_N} {(\ell-2)(1-\widehat \alpha_N)} \, \frac {\log \log N}{\log N}.
$$ 
For technical reasons ({\it i.e.} $\Pr(\widetilde
\alpha_N \leq \alpha^*)\limitN 0$), which is not satisfied by $\widehat \alpha _N$, see Bardet {\it et al.}, 2008), in the sequel we prefer to consider $\widetilde \alpha_N$ rather than $\widehat \alpha_N$.
Finally, using the usual expression of PGLSE, the adaptive estimators of $c$ and $d$ can be defined as follows:
\begin{eqnarray}\label{tilded}
\Big( \begin{array}{c} \widetilde  c_N\\ 2\widetilde  d_N\end{array}\Big ):=\big (Z_{N^{\widetilde \alpha_N}}'\widehat \Gamma_N^{-1} Z_{N^{\widetilde \alpha_N}}\big )^{-1} Z_{N^{\widetilde \alpha_N}}'\widehat \Gamma_N^{-1} Y_N(\widetilde \alpha_N).
\end{eqnarray}
The following theorem provides
the asymptotic behavior of the estimator $\widetilde  d_N$,
\begin{thm}\label{tildeD}
Under assumptions of Proposition \ref{tlclog}, with $\displaystyle \sigma^2_d(\ell):=\big ( 0 ~ \frac 1 2 \big ) \big ( Z_{1}'\big ( \Gamma(1,\cdots,\ell,d,\psi)\big )^{-1} Z_{1} \big )^{-1} \big ( 0 ~ \frac 1 2 \big )'$, 
\begin{eqnarray}\label{CLTD2}
\sqrt{\frac{N}{N^{\widetilde \alpha_N }}} \big(\widetilde d_N - d \big)
\tend \mathcal{N}(0\, ; \,\sigma^2_d(\ell))~~~\mbox{and}~~~\forall
\rho>\frac {2(1+3d')}{(\ell-2)d'},~\mbox{}~~ \frac {N^{\frac
{d'}{1+2d'}} }{(\log N)^\rho} \cdot \big|\widetilde d_N - d \big|
\limiteprobaN  0.
\end{eqnarray}
\end{thm}
\begin{rem}
\begin{enumerate}
\item From Gauss-Markov Theorem it is clear that the asymptotic variance of $\widetilde d_N$ is smaller or equal to the one of $\widehat {\widehat d_N}$. Moreover $\widetilde d_N$ satisfies the CLT (\ref{CLTD2}) which provides confidence intervals that are simple to calculate.
\item In the Gaussian case, the adaptive estimator  $\widetilde
d_N$ converge to $d$ with a rate of convergence rate equal to the
minimax rate of convergence $N^{\frac{d'}{1+2d'}}$ up to a logarithm
factor (see Giraitis {\it et al.}, 1997). Thus, this estimator can be compared to adaptive log-periodogram or local Whittle estimators (see respectively Moulines and Soulier, 2003, and Robinson, 1995).
\item Under additive assumptions on $\psi$ ($\psi$ is supposed to have its first
$m$ vanishing moments), the  estimator $\widetilde d_N$ can also be applied to a process $X$ with an additive
polynomial trend of degree $\leq m-1$. Then the tend is ``vanished'' by the wavelet function and the value of $\widetilde d_N$ is the same than without this additive trend. Such robustness property is not possible
with an adaptive log-periodogram or local Whittle estimator.
\end{enumerate}
\end{rem}
\noindent Finally it is easy to deduce from the previous pseudo-generalized least square regression an adaptive goodness-of-fit test. It consists on a sum of the PGLS squared distances between the PGLS regression line and the points. More precisely consider the statistic:
\begin{equation}\label{test}
\widetilde T_N=\frac N {N^{\widetilde \alpha_N}} \, \Big (Y_N(\widetilde \alpha_N)-  Z_{N^{\widetilde \alpha_N}} \big( \begin{array}{c} \widetilde  c_N\\ 2\widetilde  d_N\end{array}\big ) \Big )'\,  \widehat \Gamma_N^{-1} \, \Big (Y_N(\widetilde \alpha_N)-  Z_{N^{\widetilde \alpha_N}} \big( \begin{array}{c} \widetilde  c_N\\ 2\widetilde  d_N\end{array}\big ) \Big ).
\end{equation}
Then, using the previous results, one obtains:
\begin{thm}\label{tildeT}
Under assumptions of Proposition \ref{tlclog}, 
\begin{eqnarray}\label{Testconv}
\widetilde T_N
\tend \chi^2(\ell-2).
\end{eqnarray}
\end{thm}
\noindent This (adaptive) goodness-of-fit test is therefore very simple to be computed and used. In the case where $ d> 0 $, which can be tested easily from Theorem \ref{tildeD}, this test can also be seen as a test of long memory for linear processes.
ion 
\section{Simulations} \label{Simu}
In the sequel, the numerical consistency and robustness of ${\widetilde  d_N}$ are first investigated.  Simulation are realized and the results obtained with the estimator ${\widetilde  d_N}$
are compared to those obtained with the best known semiparametric long-memory estimators. Finally numerical properties of the test statistic $\widetilde T_N$ are also studied.
\begin{rem}
Note that all
the softwares (in Matlab language) used in this section are available with a free access on {\tt
http://samm.univ-paris1.fr/-Jean-Marc-Bardet}.
\end{rem}
\noindent To begin with, the simulation conditions have to be specified.
The results are obtained from $100$ generated independent samples of
each process belonging to the following "benchmark". The concrete
procedures of generation of these processes are obtained from the
circulant matrix method in case of Gaussian processes  or a truncation of an infinite sum in case of non-Gaussian process (see Doukhan {\it et al.}, 2003).
The simulations are realized for $d=0,\,0.1,\, 0.2,\, 0.3$ and $0.4$, for
$N=10^3$ and $10^4$ and the following processes which satisfy Assumption A$(d,d')$:
\begin{enumerate}
\item the fractional Gaussian noise (fGn) of parameter $H=d+1/2$ (for $0\leq d<0.5$) and $\sigma^2=1$. A fGn is such that
Assumption A$(d,2)$ holds even if a fGn is generally not studied as a Gaussian linear process;
\item a FARIMA$[p,d,q]$ process with parameter $d$ such that
$d \in [0,0.5)$, $p, q\, \in \N$.
A FARIMA$[p,d,q]$  process is such that
Assumption A$(d,2)$ holds when $\E \xi_0^4<\infty$ where $\xi_0$ is the innovation process. 
\item the Gaussian stationary process $X^{(d,d')}$, such that its spectral density is
\begin{eqnarray}
f_3(\lambda)=\frac 1 {\lambda^{2d}}(1+\lambda^{d'})~~~\mbox{for
$\lambda \in [-\pi,0)\cup (0,\pi]$},
\end{eqnarray}
with $d \in [0,0.5)$ and $d'\in (0,\infty)$. Therefore the spectral density $f_{3}$ is such that
Assumption $A(d,d')$ holds and since $X^{(d,d')}$ is a Gaussian process, from the Wold decomposition it is also a linear process.
\end{enumerate}
A "benchmark" which will be considered in the sequel consists of the following particular cases of these processes for $d=0,\,0.1,\,0.2,\,0.3,\,0.4$:
\begin{itemize}
\item $X_1:$ fGn processes with parameters $H=d+1/2$;
\item $X_2:$ FARIMA$[0,d,0]$ processes with standard Gaussian
innovations;
\item $X_3:$ FARIMA$[0,d,0]$ processes with
innovations following a uniform ${\cal U}[-1,1]$ distribution;
\item $X_4:$ FARIMA$[0,d,0]$ processes  with
innovations following a symmetric Burr distribution of parameter $(2,1)$ ({\it i.e.} its cumulative distribution function is $F(x)=(1-\frac 1 2 \, (1+x^2)^{-1})\,\I_{x\geq 0}+\frac 1 2 \,(1+x^2)\, \I_{x<0}$);
\item $X_5:$ FARIMA$[0,d,0]$ processes  with
innovations following a Cauchy distribution;
\item $X_6:$ FARIMA$[1,d,1]$ processes with standard Gaussian
innovations, MA coefficient $\phi=-0.3$ and AR coefficient
$\phi=0.7$;
\item $X_7:$ FARIMA$[1,d,1]$ processes with
innovations following a uniform ${\cal U}[-1,1]$ distribution, MA coefficient $\phi=-0.3$ and AR coefficient
$\phi=0.7$;
\item $X_8:$ $X^{(d,d')}$ Gaussian processes with $d'=1$.
\end{itemize}
Note that the processes $X_4$ and $X_5$ do not satisfy the condition $\E \xi_0^4$ required in Theorems \ref{tildeD} and \ref{tildeT}.
However, since we consider the logarithm of wavelet coefficient sample variance and not only the wavelet coefficient sample variance, it should be possible to prove the consistency of $\widetilde d_N$ under a condition such as $\E \xi_0^r$ with $r\geq 2$ and perhaps only $r>0$...
\subsection{Comparison of the wavelet based estimator and other estimators}
First let us specify the different choices concerning the wavelet based estimator:\\
~\\
{\bf Choice of the function $\psi$:} as it was said previously, it is not mandatory to use a wavelet function associated with a multi-resolution analysis. We use here the function $\psi(x)=x^3(1-x)^3\big (x^3-\frac 3 2 \, x^2+\frac {15}{22}\, x-\frac 1 {11}\big )\I_{x\in [0,1]}$ which satisfies Assumption $\Psi(2)$  \\
~\\
{\bf Choice of the parameter $\ell$:} This parameter is
important to estimate the "beginning" of the linear part of the
graph drawn by points $(\log (ia_N), \log T_N(ia_N))_{1\leq i \leq \ell}$ and therefore the data-driven $\widehat a_N$. Moreover this parameter is used for the computation of $\tilde d_N$ as the number of regression points. We chose a two step procedure:
\begin{enumerate}
\item following a numerical study (not detailed here), $\ell=[2*log(N)]$ (therefore $\ell=13$ for $N=1000$ and $\ell=18$ for $N=10000$) seems to be a good choice for the first step: compute $\widehat \alpha_n$. 
\item for the computation of $\tilde d_N$, we first remark that with the chosen function $\psi$, $\widehat \Gamma_N$ does not seem to depend on $d$. As a consequence we decide to compute $\sigma^2_d(\ell)=\big ( 0 ~ \frac 1 2 \big ) \big ( Z_{1}'\big ( \Gamma(1,\cdots,\ell,d,\psi)\big )^{-1} Z_{1} \big )^{-1} \big ( 0 ~ \frac 1 2 \big )'$ for several values of $d$ and $\ell$ using classical approximations of the integrals defined in $\Gamma(1,\cdots,\ell,d,\psi)$. The results of these numerical experiments are reported in Figure \ref{Figure1}. The conclusion of this numerical experiment is the following: for any $d \in [0,0.5 [$, $\sigma^2_d(\ell)$ is almost not depending on $d$ and decreases when $\ell$ increases. Therefore we chose for this second step $\ell=N^{1-\widetilde \alpha_N}(\log N)^{-1}$: by this way the larger considered scale is $N(\log N)^{-1}$ (which is negligible with respect to $N$ and therefore the CLT \ref{CLTSN} holds). 
\begin{figure}[ht]\label{Figure1}
\[
\epsfxsize 13cm \epsfysize 6cm \epsfbox{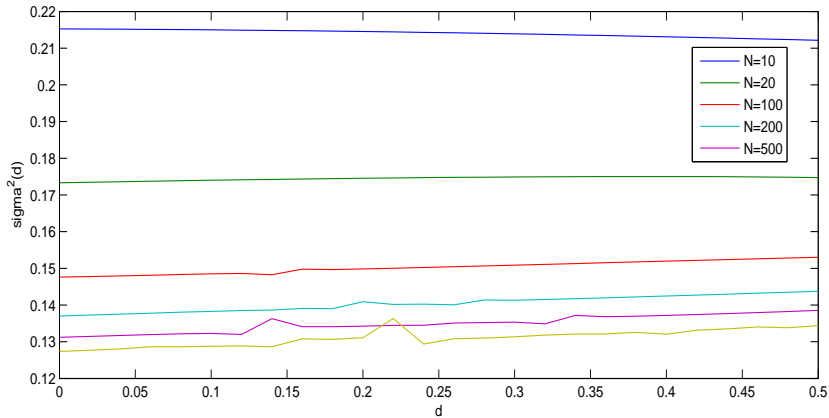}
\]
\caption{\it Graph of the approximated values of $\sigma^2_d(\ell)$ defined in \ref{CLTD2} for $d \in [0,0.5]$ and $N=10,20,50,100,200$ and $500$.}
\end{figure}
\end{enumerate}
Now we consider the previous "benchmark" of processes and apply the estimator  $ {\widetilde d_N}$ and $2$ other
 semiparametric $d$-estimators known for their accuracies:
\begin{itemize}
\item $\widehat d_{MS}$ is the adaptive global log-periodogram estimator introduced
by Moulines and Soulier (1998, 2003), also called FEXP estimator,
with bias-variance balance parameter $\kappa=2$;
\item $\widehat d_{R}$ is the local Whittle estimator introduced by
Robinson (1995). The trimming parameter is $m=N/30$.
\end{itemize}
Simulation results are reported in {Table} {\ref{Table4}}. \\
\begin{table}[p] {\scriptsize
\begin{center}
$N=10^3~\longrightarrow $ ~\begin{tabular}{|c|c|c|c|c|c|c|}
\hline\hline
Model & $\sqrt{MSE}$ & $d=0$ & $d=0.1$ & $d=0.2$ & $d=0.3$ & $d=0.4$ \\
\hline \hline
$X_1$& $\sqrt{MSE}$ $\widehat d_{MS}$ & 0.089 &0.091&0.096 &0.090&0.100 \\
&  $\sqrt{MSE}$ $\widehat d_{R}$ & 0.102 &0.114&0.116 &0.106&0.102\\
&  $\sqrt{MSE}$ $ {\widetilde d_N}$ &\bf 0.047  &\bf 0.046&\bf 0.042 &\bf 0.052&\bf 0.047 \\
&  $\widetilde {p}_n $ &0.85 & 0.76&0.78 &0.76 & 0.64\\
\hline  $X_2$& $\sqrt{MSE}$ $\widehat d_{MS}$ & 0.091 &0.094&0.086 &0.091&0.099 \\
&  $\sqrt{MSE}$ $\widehat d_{R}$ & 0.107 &0.105&0.112 &0.110&0.097 \\
&  $\sqrt{MSE}$ $ {\widetilde d_N}$ &\bf 0.048 &\bf0.050&\bf 0.053 &\bf 0.061&\bf 0.074\\
&  $\widetilde {p}_n $ &0.82 & 0.82&0.75 &0.73 & 0.67\\
\hline $X_3$& $\sqrt{MSE}$ $\widehat d_{MS}$ & 0.092 &0.094&0.080 &0.099&0.096 \\
&  $\sqrt{MSE}$ $\widehat d_{R}$ & 0.113 &0.113&0.100 &0.112&0.095\\
&  $\sqrt{MSE}$ $ {\widetilde d_N}$ &\bf 0.052  &\bf 0.071&\bf 0.063 &\bf 0.077 & \bf 0.092 \\
&  $\widetilde {p}_n $ &0.84 & 0.72&0.75 &0.67 & 0.51\\
\hline  $X_4$& $\sqrt{MSE}$ $\widehat d_{MS}$ & 0.088 &0.079&0.079 &0.093&0.104 \\
&  $\sqrt{MSE}$ $\widehat d_{R}$ & 0.096 &0.100&0.103 &0.097&0.095 \\
&  $\sqrt{MSE}$ $ {\widetilde d_N}$ &\bf  0.051 &\bf 0.066&\bf  0.056 &\bf 0.061&\bf 0.064\\
&  $\widetilde {p}_n $ &0.84 & 0.78&0.78 &0.75 & 0.66\\
\hline $X_5$& $\sqrt{MSE}$ $\widehat d_{MS}$ &\bf  0.069 &\bf 0.067&\bf 0.077 &0.121&0.143 \\
&  $\sqrt{MSE}$ $\widehat d_{R}$ & 0.072 &0.078&0.093 &\bf 0.087& \bf 0.074\\
&  $\sqrt{MSE}$ $ {\widetilde d_N}$ & 0.073  & 0.069& 0.083 & \bf 0.087&0.120 \\
&  $\widetilde {p}_n $ &0.73 & 0.69&0.68 &0.74 & 0.64\\
\hline  $X_6$& $\sqrt{MSE}$ $\widehat d_{MS}$ &\bf  0.096 &\bf 0.091&\bf 0.090 &\bf 0.086&\bf 0.093 \\
&  $\sqrt{MSE}$ $\widehat d_{R}$ & 0.111 &0.102&0.100 &0.101&0.101 \\
&  $\sqrt{MSE}$ $ {\widetilde d_N}$ & 0.153 &0.146& 0.144 & 0.158& 0.147\\
&  $\widetilde {p}_n $ &0.52 & 0.47&0.48 &0.39 & 0.50\\
\hline $X_7$& $\sqrt{MSE}$ $\widehat d_{MS}$ &\bf 0.085 &\bf 0.096&\bf 0.086 &\bf 0.093& 0.098 \\
&  $\sqrt{MSE}$ $\widehat d_{R}$ & 0.106 &0.116&0.097 &0.099&\bf 0.092\\
&  $\sqrt{MSE}$ $ {\widetilde d_N}$ & 0.155 & 0.150& 0.56 & 0.147&0.157 \\
&  $\widetilde {p}_n $ &0.60 & 0.55&0.49 &0.52 & 0.41\\
\hline  $X_8$& $\sqrt{MSE}$ $\widehat d_{MS}$ &\bf  0.097 &\bf 0.104& \bf0.097 & \bf 0.094&\bf 0.101\\
&  $\sqrt{MSE}$ $\widehat d_{R}$ & 0.120 &0.116&0.117 &0.113&0.110  \\
&  $\sqrt{MSE}$ $ {\widetilde d_N}$ & 0.179 &0.189& 0.177 & 0.175& 0.176\\
&  $\widetilde {p}_n $ &0.75 & 0.75&0.68 &0.66 & 0.67\\
\hline \end{tabular}
\end{center}}
{\scriptsize
\begin{center}
$N=10^4~\longrightarrow $ ~\begin{tabular}{|c|c|c|c|c|c|c|}
\hline\hline
Model & $\sqrt{MSE}$ & $d=0$ & $d=0.1$ & $d=0.2$ & $d=0.3$ & $d=0.4$ \\
\hline \hline
$X_1$& $\sqrt{MSE}$ $\widehat d_{MS}$ & 0.032 &0.029&0.031 &0.031&0.036 \\
&  $\sqrt{MSE}$ $\widehat d_{R}$ & 0.028 &0.028&\bf 0.029 &0.029&0.032\\
&  $\sqrt{MSE}$ $ {\widetilde d_N}$ &\bf 0.016  &\bf 0.027& 0.034 &\bf 0.025&\bf 0.022 \\
&  $\widetilde {p}_n $ &0.97 & 0.93&0.97 &0.94 & 0.97\\
\hline  $X_2$& $\sqrt{MSE}$ $\widehat d_{MS}$ & 0.034 &0.030&0.029 &0.032&0.028 \\
&  $\sqrt{MSE}$ $\widehat d_{R}$ & 0.027 &0.027&0.029 &0.028&\bf 0.023 \\
&  $\sqrt{MSE}$ $ {\widetilde d_N}$ &\bf 0.026 &\bf0.019&\bf 0.019 &\bf 0.019& 0.025\\
&  $\widetilde {p}_n $ &0.95 & 0.97&0.98 &0.96 & 0.94\\
\hline $X_3$& $\sqrt{MSE}$ $\widehat d_{MS}$ & 0.034 &0.034&0.033 &0.030&0.031 \\
&  $\sqrt{MSE}$ $\widehat d_{R}$ & 0.029 &0.028&0.028 &0.028&\bf 0.029\\
&  $\sqrt{MSE}$ $ {\widetilde d_N}$ &\bf 0.027  &\bf 0.017&\bf 0.016 &\bf 0.022 & 0.030 \\
&  $\widetilde {p}_n $ &0.93 & 0.96&0.97 &0.93 & 0.92\\
\hline  $X_4$& $\sqrt{MSE}$ $\widehat d_{MS}$ & 0.029 &0.060&0.036 &0.031&0.031 \\
&  $\sqrt{MSE}$ $\widehat d_{R}$ & 0.025 &0.027&0.029 &0.031&0.029 \\
&  $\sqrt{MSE}$ $ {\widetilde d_N}$ &\bf  0.016 &\bf 0.020&\bf  0.021 &\bf 0.015&\bf 0.023\\
&  $\widetilde {p}_n $ &0.95 & 0.91&0.97 &0.92 & 0.91\\
\hline $X_5$& $\sqrt{MSE}$ $\widehat d_{MS}$ &  0.093 & \bf 0.046& 0.039 &0.073&0.047 \\
&  $\sqrt{MSE}$ $\widehat d_{R}$ & \bf 0.040 &\bf 0.046&0.035 & 0.032& \bf 0.024\\
&  $\sqrt{MSE}$ $ {\widetilde d_N}$ & 0.056  & 0.071&\bf  0.027 & \bf 0.025 &\bf 0.024 \\
&  $\widetilde {p}_n $ &0.85 & 0.88&0.93 &0.86 & 0.85\\
\hline  $X_6$& $\sqrt{MSE}$ $\widehat d_{MS}$ & 0.031 & 0.032& 0.033 &0.032& 0.029 \\
&  $\sqrt{MSE}$ $\widehat d_{R}$ &\bf 0.029 &\bf 0.028&\bf 0.028 &\bf 0.028&\bf 0.028 \\
&  $\sqrt{MSE}$ $ {\widetilde d_N}$ & 0.045 &0.044& 0.046 & 0.044& 0.041\\
&  $\widetilde {p}_n $ &0.96 & 0.93&0.89 &0.93 & 0.90\\
\hline $X_7$& $\sqrt{MSE}$ $\widehat d_{MS}$ & 0.030 & 0.031& 0.037 & 0.030& 0.029 \\
&  $\sqrt{MSE}$ $\widehat d_{R}$ &\bf 0.027 &\bf 0.027&\bf 0.032 &\bf 0.028&\bf 0.027\\
&  $\sqrt{MSE}$ $ {\widetilde d_N}$ & 0.049 & 0.044& 0.050 & 0.048&0.046 \\
&  $\widetilde {p}_n $ &0.94 & 0.91&0.88 &0.87 & 0.86\\
\hline  $X_8$& $\sqrt{MSE}$ $\widehat d_{MS}$ &\bf  0.038 & 0.040& \bf0.040 & \bf 0.035& 0.037\\
&  $\sqrt{MSE}$ $\widehat d_{R}$ & 0.039 &\bf 0.038&\bf 0.040 &0.036&\bf 0.035  \\
&  $\sqrt{MSE}$ $ {\widetilde d_N}$ & 0.085 &0.083& 0.086 & 0.087& 0.085\\
&  $\widetilde {p}_n $ &0.92 & 0.94&0.94 &0.95 & 0.93\\
\hline 
\end{tabular}
\end{center}}
\caption {\label{Table4} \it Comparison of the different long-memory parameter estimators for processes of the benchmark.
For each process and value of $d$ and $N$, $\sqrt{MSE}$ are computed from $100$ independent generated samples. Here $\widetilde {p}_n=\frac 1 n \, \# \big (\widetilde T_N<q_{\chi^2(\ell-2)}(0.95)\big )$: this is the frequency of acceptation of the adaptive goodness-of-fit test.}
\end{table}
\newline
\noindent {\it Conclusions from Table \ref{Table4}:} The wavelet based estimator $ {\widetilde d_N}$ numerically shows a convincing convergence rate with respect to the other estimators. Both the ``spectral'' estimator $\widehat d_R$ and $\widehat d_{MS}$ provide more stable results almost not sensible to $d$ and the flatness of the spectral density of the process, while the convergence rate of the wavelet based estimator ${\widetilde d_N}$ is more dependent on the spectral density of the process. But, especially in cases of ``smooth'' spectral densities (fGn and FARIMA$(0,d,0)$), $ {\widetilde d_N}$ is a very accurate semiparametric estimator and is globally more efficient than the other estimators.
\begin{rem}
In Bardet {\it et al.} (2008) we also compared two adaptive wavelet based estimators (the one defined in Veitch {\it et al.}, 2003 and the one defined in Bardet {\it et al.}, 2008) with  $\widehat d_{MS}$  and $\widehat d_{R}$ (and also with two others defined in Giraitis {\it et al.}, 2000, and Giraitis {\it et al.}, 2006, which exhibit worse numerical properties of consistency). We observe that $\sqrt{MSE}$ of $ {\widetilde d_N}$ obtained in Table \ref{Table4} is generally smaller than the one obtained with the estimator defined in Bardet {\it et al.} (2008) for two reasons: the choice of the definition (\ref{samplevar}) of wavelet coefficient sample variance instead of (\ref{samplevarusual}) and the choice of a PGLS regression instead of a LS regression. 
\end{rem}

\noindent{\bf Comparison of the robustness of the different semiparametric estimators:} To
conclude with the  numerical properties of the estimators, $3$ different
processes not satisfying Assumption $A(d,d')$ are considered:
\begin{itemize}
\item a Gaussian stationary process with a spectral density $f(\lambda)=\big ||\lambda|-\pi/2\big |^{-2d}$
for all $\lambda\in [-\pi,\pi] \setminus \{-\pi/2,\pi/2\}$. The
local behavior of $f$ in $0$ is $f(|\lambda|) \sim (\pi/2)^{-2d}\,
|\lambda|^{-2d}$ with $d=0$, but the smoothness condition for $f$ in
Assumption $A(0,2)$ is not satisfied.
\item a trended Gaussian FARIMA$(0,d,0)$ with an additive linear trend ($X_t=FARIMA_t+(1-2t/n)$ for $t=1,\cdots,n$ and therefore $mean(X_1,\cdots,X_n)\simeq 0$);
\item a Gaussian FARIMA$(0,d,0)$ with an additive linear trend	and an additive sinusoidal seasonal component of period $T=12$ ($X_t=FARIMA_t+(1-2t/n)+\sin(\pi \,t/6)$ for $t=1,\cdots,n$ and therefore $mean(X_1,\cdots,X_n)\simeq 0$).
\end{itemize}
The results of these simulations are given in {Table} {\ref{Table5}}.\\
\begin{table}[t] {\scriptsize
\begin{center}
$N=10^3~\longrightarrow $ ~\begin{tabular}{|c|c|c|c|c|c|c|}
\hline\hline
Model & $\sqrt{MSE}$& $d=0$ & $d=0.1$ & $d=0.2$ & $d=0.3$ & $d=0.4$ \\
\hline \hline
\hline  GARMA$(0,d,0)$ &  $\sqrt{MSE}$ $\widehat d_{MS}$ & 0.089 &0.091& 0.123 & 0.132&0.166 \\
&  $\sqrt{MSE}$ $\widehat d_{R}$ & 0.112 &0.111&0.119 &\bf 0.106& \bf 0.106 \\
&  $\sqrt{MSE}$ ${\widetilde d_N}$ &\bf 0.041 &\bf 0.076&\bf 0.114 &0.142&0.180 \\
&  $\widetilde {p}_n $ &0.82 & 0.78&0.63 &0.59 & 0.46\\
\hline  Trend &  $\sqrt{MSE}$ $\widehat d_{MS}$ & 0.548 &0.411&  0.292 &  0.190& 0.142 \\
&  $\sqrt{MSE}$ $\widehat d_{R}$ & 0.499 &0.394&0.279 &0.167&0.091 \\
&  $\sqrt{MSE}$ ${\widetilde d_N}$ &\bf 0.044 & \bf 0.052&\bf  0.056 &\bf 0.060& \bf 0.065 \\
&  $\widetilde {p}_n $ &0.83 & 0.81&0.80 &0.73 & 0.64\\
\hline Trend + Seasonality &  $\sqrt{MSE}$ $\widehat d_{MS}$ & 0.479 &0.347& 0.233 & 0.142& 0.112 \\
&  $\sqrt{MSE}$ $\widehat d_{R}$ & 0.499 &0.393&0.279 &0.167&\bf 0.091 \\
&  $\sqrt{MSE}$ $ {\widetilde d_N}$ &\bf 0.216 &\bf 0.215& \bf 0.215 &\bf 0.217& 0.185 \\
&  $\widetilde {p}_n $ &0.35 & 0.26&0.18 &0.21 & 0.18\\
\hline
\end{tabular}
\end{center}}
{\scriptsize
\begin{center}
$N=10^4~\longrightarrow $ ~\begin{tabular}{|c|c|c|c|c|c|c|}
\hline\hline
Model & $\sqrt{MSE}$& $d=0$ & $d=0.1$ & $d=0.2$ & $d=0.3$ & $d=0.4$ \\
\hline \hline
\hline  GARMA$(0,d,0)$ &  $\sqrt{MSE}$ $\widehat d_{MS}$ & 0.031 &0.035& 0.039 & 0.049&0.062 \\
&  $\sqrt{MSE}$ $\widehat d_{R}$ & 0.028 &\bf 0.031&\bf 0.030 &\bf 0.030& \bf 0.034 \\
&  $\sqrt{MSE}$ ${\widetilde d_N}$ &\bf 0.023 & 0.053& 0.052 &0.058&0.060 \\
&  $\widetilde {p}_n $ &0.96 & 0.94&0.93 &0.91 & 0.88\\
\hline  Trend &  $\sqrt{MSE}$ $\widehat d_{MS}$ & 0.452 &0.286&  0.167 &  0.096& 0.056 \\
&  $\sqrt{MSE}$ $\widehat d_{R}$ & 0.433 &0.308&0.191 &0.100&0.051 \\
&  $\sqrt{MSE}$ ${\widetilde d_N}$ &\bf 0.014 & \bf 0.016&\bf  0.016 &\bf 0.021& \bf 0.028 \\
&  $\widetilde {p}_n $ &0.99 & 0.97&0.97 &0.95 & 0.93\\
\hline Trend + Seasonality &  $\sqrt{MSE}$ $\widehat d_{MS}$ & 0.471 &0.307& 0.196 & 0.123& 0.076 \\
&  $\sqrt{MSE}$ $\widehat d_{R}$ & 0.432 &0.305&0.191 &0.100& 0.052 \\
&  $\sqrt{MSE}$ $ {\widetilde d_N}$ &\bf 0.044 &\bf 0.069& \bf 0.047 &\bf 0.042& \bf 0.045 \\
&  $\widetilde {p}_n $ &0.83 & 0.81&0.76 &0.78 & 0.82\\
\hline
\end{tabular}
\end{center}}
\caption {\label{Table5} \it Robustness of the different long-memory parameter estimators. For each process and value of $d$ and $N$, $\sqrt{MSE}$ are computed from $100$ independent generated samples. Here $\widetilde {p}_n=\frac 1 n \, \# \big (\widetilde T_N<q_{\chi^2(\ell-2)}(0.95)\big )$: this is the frequency of acceptation of the adaptive goodness-of-fit test.}
\end{table}
\newline
\noindent ~\\
{\it Conclusions from {Table} {\ref{Table5}}:}  The main advantages of $ {\widetilde d_N}$ with respect to $\widehat d_{MS}$ and $\widehat d_{R}$ are exhibited in this table: it is robust with respect to smooth trends (or seasonality). Note that the sample mean of $\widehat d_{MS}$ and $\widehat d_{R}$ in the case of processes with trend or with trend and seasonality is almost $0.5$
\subsection{Consistency and robustness of the adaptive goodness-of-fit test:}
Tables {\ref{Table4}} and {\ref{Table5}} provide informations concerning the adaptive goodness-of-fit test. A general conclusion is that the consistency properties of this test are clearly satisfying when $N$ is large enough ($N=1000$ seems to be too small for using this goodness-of-fit test). \\
~\\ 
We also would like to know the behavior of the test statistic under the assumption $H_1$. We are going to study the case of a process which does not satisfy either the stationarity condition either the relation (\ref{alpha}) also verified by the spectral density. Hence $3$ particular cases are considered:
\begin{enumerate}
\item a process $X$ denoted MFARIMA and defined as a succession of two independent Gaussian FARIMA processes. More precisely, we consider $X_t=FARIMA(0,0.1,0)$ for $t=1,\cdots,n/2$ and $X_t=FARIMA(0,0.4,0)$ for $t=n/2+1,\cdots,n$. 
\item a process $X$ denoted MGN and defined by the increments of a multifractional Brownian motion (introduced in Peltier and Lévy-Vehel, 1995). Using the harmonizable representation, define $Y=(Y_t)_t$ such that 
$$ 
Y_t  = C(t) \  \int_{\R} \frac{{e}^{{i}t x} -
1}{|x|^{H(t)+1/2}} W(d x)$$ 
where $W(d x)$ is a complex-valued Gaussian noise with variance $d
x$ and $H(\cdot)$ is a function (the case $H(\cdot)= H$ with $H\in (0,1)$ is the case of fBm), $C(\cdot)$ i a function.  Here we consider the functions $H(t)=0.5+0.4\sin(t/10)$ and $C(t)=1$. Then $X_t=Y_{t+1}-Y_t$ for $t\in \Z$. $X$ is not a stationary process but ``locally'' behaves as a fGn with a parameter $H(t)$ (therefore depending on $t$).
\item a process $X$  denoted MFGN and defined by the increments of a multiscale fractional Brownian motion (introduced in Bardet and Bertrand, 2007). Let $Z=(Z_t)_t$ be such that 
$$ 
Z_t =  \  \int_{\R} \sigma(x)\, \frac{{e}^{{i}t x} -
1}{|x|^{H(x)+1/2}} W(d x)$$ 
where $W(d x)$ is a complex-valued Gaussian noise with variance $d
x$, $H(\cdot)$ and $\sigma(\cdot)$ are piecewise constant functions.  Here we consider the functions $H(x)=0.9$ for $0.001\leq x \leq 0.04$ and $H=0.1$ for $0.04\leq x \leq 3$. Then $X_t=Z_{t+1}-Z_t$ for $t\in \Z$ and $X$ is a Gaussian stationary process which can be written as a linear process behaving as a fGn of parameter $0.9$ for low frequencies (large time) and as a fGn of parameter $0.1$ for high frequencies (small time).  
\end{enumerate}
We applied the test statistic to $100$ independent replications of both these processes. The results of this simulation are proposed in Table \ref{Table6}. We observed that the processes MGN and MFGN are clearly rejected with the adaptive goodness-of-fit test. However, the test is not able to reject the  process MFARIMA which does not satisfy the Assumption of the Theorem \ref{tildeT}. The reason is that the test does an average of the behavior of the sample and in the case of changes (it is such the case for MFARIMA) it is the average LRD parameter which is estimated (an average of  $0.30$ for $\widetilde{d_N}$ and a standard deviation $0.03$ are obtained).   \\
\begin{table}[t] {\scriptsize
\begin{center}
\begin{tabular}{|c|c|c|}
\hline\hline
Model &  $N=10^3$ &$N=10^4$ \\
\hline 
\hline 
MFARIMA & $\widetilde {p}_n =0.58$&$\widetilde {p}_n =0.87$ \\
 MGN &  $\widetilde {p}_n =0.18$&$\widetilde {p}_n =0.08$ \\
 MFGN &  $\widetilde {p}_n =0.02$&$\widetilde {p}_n =0.04$ \\
\hline
\end{tabular}
\end{center}}
\caption {\label{Table6} \it Robustness of the adaptive goodness-of-fit test with $\widetilde {p}_n=\frac 1 n \, \# \big (\widetilde T_N<q_{\chi^2(\ell-2)}(0.95)\big )$ the frequency of acceptation of the adaptive goodness-of-fit test.}
\end{table}
\newline 

\section{Proofs} \label{Proofs}
First, we will use many times the following lemma:
\begin{lem}\label{dev}
If $\psi$ satisfies Assumption $\Psi(k)$ with $k\geq 1$, then there exists $C_\psi\geq 0$ such that for all $\lambda \in  \R$,
\begin{eqnarray}\label{induction}
\Big |\frac 1 a \sum_{k=1}^a \psi\big ( \frac k a \big )e^{-i\lambda\, \frac k a} -\int_0^1
\psi(t)e^{-i \lambda\, t}dt \Big |  \leq C_{\psi} \, \frac {(1+|\lambda|^k)} {a^{k}}.
\end{eqnarray}
\end{lem}
\begin{proof}[Proof of Lemma \ref{dev}]
This proof is easily established from a mathematical induction on $k$ when $\lambda \in [-\pi,\pi]$. Then since we consider $2\pi$-periodic functions (of $\lambda$) the result can be extended to $\R$.
\end{proof}
\begin{proof}[Proof of Property \ref{cor1}]
First, it is clear that for $a\in \N^*$, $(e(a,b))_{1\leq b \leq N-a}$ is a centered linear process. It is a stationary process because $X$ is a stationary process and clearly $\sum_{k=1} ^N \frac 1 a \psi^2\big (\frac {k-b}a \big )<\infty$. \\
Now following similar computations to those performed in  Bardet {\it et al.} (2008), we obtain for $a\in \N^*$,
\begin{eqnarray}
\nonumber \E (e^{2}(a,0))  &=& \int_{-a\pi}^{a\pi} f\big ( \frac  u a \big )
\times \Big |\frac 1 {a}\sum_{k=1}^a \psi\big ( \frac  k a
\big ) e^{-i \frac  k a u}\Big |^2 \, du.
\end{eqnarray}
Now, since $\sup_{u\in \R} |\widehat \psi (u) | <\infty$, for $a$ large enough,
\begin{eqnarray}
\nonumber
\Big |\E (e^{2}(a,0))-\int_{-a\pi}^{a\pi} f\big ( \frac u
a \big ) \times|\widehat \psi (u)|^2\, du \Big | &\leq &
2\, \sup_{u\in \R} |\widehat \psi (u) | C_{\psi}\, \frac 1 {a^{k}} \int _{-a\pi}^{a\pi} (1+|u|^k) \, f\big (
\frac  u a \big ) |\widehat \psi(u)|^2\, du
\end{eqnarray}
Under Assumption $\Psi(k)$ for any $k\geq 1$, $\sup_{u\in \R} (1+u^n)|\widehat \psi (u)| <\infty$ for all $n
\in \N$. Therefore, since there exists $c_a>0$ satisfying $f(\lambda)\leq c_a \lambda^{-2d}$ for all $\lambda\in [-\pi,\pi]$,  for all $d<1/2$,
$$
\int _{-a\pi}^{a\pi} (1+|u|^k) \, f\big (
\frac  u a \big ) |\widehat \psi(u)|^2\, du \leq \Big ( c_a\int_{-\infty}^\infty (1+|u|^k) u^{-2d} |\widehat \psi(u)|^2\, du\Big )\, a^{2d},
$$
and thus  there exists $C>0$ (not depending on $a$) such that for $a$ large enough,
\begin{eqnarray}\label{Fourier2}
\Big |\E (e^{2}(a,0))-\int_{-a\pi}^{a\pi} f\big ( \frac u a \big )
\times|\widehat \psi (u)|^2\, du \Big | &\leq &C\, a^{2d-k}.
\end{eqnarray}
Following the same reasoning, for any $n \geq 0$, there exists $C(n)>0$ (not depending on $a$) such that for $a$ large enough,
\begin{eqnarray}\label{Fourier3}
\Big |\int_{-a\pi}^{a\pi} f\big ( \frac u a \big )
\, |\widehat \psi (u)|^2\, du -\int_{-\infty}^{\infty} f\big ( \frac u a \big )
\, |\widehat \psi (u)|^2\, du\Big | \leq C(n)\, a^{-n}.
\end{eqnarray}
Finally, from Assumption A$(d,d')$, we obtain the following expansion:
\begin{eqnarray}
\nonumber \int_{-\infty}^{\infty} f\big ( \frac u a \big )
\, |\widehat \psi (u)|^2\, du &= &2\pi \, \int_{-\infty}^{\infty} \big (c_d( \frac u a \big )^{-2d}+c_{d'}( \frac u a \big )^{d'-2d}+( \frac u a \big )^{d'-2d}\varepsilon(\frac u a) \big ) 
 \,|\widehat \psi (u)|^2\, du\\
\label{Fourier4}&= &2\pi \, c_d  \,  K_{(\psi,2d)}\, a ^{-2d}+2\pi \, c_{d'}\, K_{(\psi,2d-d')}\, a ^{2d-d'}+o(a ^{2d-d'})
 \end{eqnarray}
using the definition (\ref{Kpsi}) of $K_{(\psi,\alpha)}$ and because $\lim_{\lambda \to 0} \varepsilon(\lambda)=0$ and applying Lebesgue Theorem. 
Then, using (\ref{Fourier2}), (\ref{Fourier3}) and (\ref{Fourier4}), we obtain that 
\begin{eqnarray}\label{inegdetal}
\Big |\E (e^{2}(a,0))-2\pi \, c_d  \,  K_{(\psi,2d)}\, a ^{-2d}+2\pi \, c_{d'}  \, K_{(\psi,2d-d')}\, a ^{2d-d'}\Big |  & \leq & o(a ^{2d-d'})+ C\, a^{2d-k}.
\end{eqnarray}
When $k> d'$, it implies (\ref{equiD'}).  
\end{proof}
\begin{proof}[Proof of Theorem \ref{tlclog}]
We decompose this proof in $4$ steps. First define the normalized  wavelet coefficients of $X$ by:
\begin{equation}\label{dtilde}
\widetilde{e}_N(a,b):=\frac{e(a,b)}{\sqrt {\E e^2(a,1)}} \quad \mbox{for $a\in \N^*$ and $b\in \Z$,}
\end{equation}
and the normalized sample variance of wavelet coefficients:
\begin{eqnarray} \label{samplevar2}
\widetilde{T}_{N}(a):=\frac{1}{N-a}\sum_{k=1}^{N-a}\widetilde{e}^{2}(a,k).
\end{eqnarray}
\\{\bf Step 1}
We prove in  this part that 
$\displaystyle N  \Cov (\widetilde T_N(r\,a_N), \widetilde
T_N(r'\,a_N))$  converges to the asymptotic covariance matrix $\Gamma(\ell_1,\cdots,r_\ell,\psi,d)$ defined in (\ref{cov}).  First for $\lambda \in \R$, denote 
$$
S_a(\lambda):= \frac 1 a \, \sum_{t=1}^{a}\psi(\frac{t}{a})e^{i\lambda t/a}.
$$
Then for $a\in \N^*$ and $b=1,\cdots,N-a$, since $\psi$ is $(0,1])$-supported function,
\begin{eqnarray}
\sum_{t=1}^{N}\alpha(t-s)\psi(\frac{t-b}{a})&=& \sum_{t=0}^{a} \psi\big (\frac {t}a\big) \int_{-\pi}^{\pi}\widehat \alpha (\lambda)e^{i\lambda (t-s+b)}d\lambda \nonumber\\
&=&\int_{-\pi}^{\pi} a S_a(a\lambda) \widehat \alpha (\lambda) e^{i(b-s)\lambda} d\lambda\nonumber \\
&=&\int_{-a\pi}^{a\pi} S_a(\lambda) \widehat \alpha (\frac \lambda a) e^{i(b-s)\frac \lambda a } d\lambda. \label{trans}
\end{eqnarray} 
Thus,
\begin{eqnarray}
\nonumber \Cov (\widetilde T_N(a), \widetilde T_N(a'))&= &\frac 1
{N-a}\frac 1 {N-a'} \sum_{b=1}^{N-a}
\sum_{b'=1}^{N-a'}\Cov(\widetilde{e}^2(a,b),\widetilde{e}^2(a',b')) \\
\label{covSN} & \simeq & \frac {(a\,a')^{-2d} (c_d \, K_{(\psi,2d)} )^{-2}}
{4\pi^2(N-a)(N-a')} \sum_{b=1}^{N-a}
\sum_{b'=1}^{N-a'} \Cov({e}^2(a,b),{e}^2(a',b')).
\end{eqnarray}
But,
\begin{eqnarray}
\nonumber \Cov({e}^{2}_{(a,b)},{e}^{2}_{(a',b')}) \hspace{-3mm} &=&\hspace{-3mm}\frac 1 {a \, a'} \sum_{t_1,t_2,t_3,t_4=1}^N \sum_{s_1,s_2,s_3,s_4 \in \Z}
\Big (\prod_{i=1}^2 
\alpha(t_{i}-s_{i})\psi(\frac{t_{i}-b}{a}) \Big ) \Big (\prod_{i=1}^2 
\alpha(t_{i}-s_{i})\psi(\frac{t_{i}-b'}{a'}) \Big ) \Cov\big(\xi_{s_{1}}\xi_{s_{2}},\xi_{s_{3}}\xi_{s_{4}}\big)\\
&= & C_1 +C_2,
\end{eqnarray}
since there are only two nonvanishing cases: $s_1=s_2=s_3=s_4$ (Case 1 $=>C_1$), $s_1=s_3 \neq s_2=s_4$  and  $s_1=s_4 \neq s_2=s_3$ (Case 2 $=>C_2$).\\
* {\it Case 1:} in such a case, $\Cov\big(\xi_{s_{1}}\xi_{s_{2}},\xi_{s_{3}}\xi_{s_{4}}\big)=\mu_4-1$ and 
\begin{eqnarray*}
\nonumber
C_1 \hspace{-3mm}&=& \hspace{-3mm}\frac{ \mu_4-1} {a \, a'} \sum_{s\in \Z} \Big |\sum_{t=1}^N 
\alpha(t-s)\psi(\frac{t-b}{a}) \Big |^2  \Big | \sum_{t=1}^N 
\alpha(t-s)\psi(\frac{t-b'}{a'}) \Big |^2 \\
C_1 \hspace{-3mm} &= & \hspace{-3mm} (\mu_4-1)\, a \, a'\, \lim_{M\to \infty} \int_{[-\pi,\pi]^4}\hspace{-0.7cm}d\lambda d\lambda' d\mu d\mu' e^{i
 [b(\lambda-\lambda')+b'(\mu-\mu')]} \\
&& \hspace{3cm}\times \sum_{s=-M}^M e^{is [(\lambda-\lambda')+(\mu-\mu')]}S_{a}(a\lambda)\widehat \alpha (\lambda)\overline{S_{a}(a\lambda')}\overline{\widehat \alpha (\lambda')}S_{a'}(a'\mu)\widehat \alpha (\mu)\overline{S_{a'}(a'\mu')}\overline{\widehat \alpha (\mu')}
\end{eqnarray*}
using the relation (\ref{trans}). From usual asymptotic behavior of Dirichlet kernel, for $g \in {\cal C}_{2\pi}^1((-\pi,\pi))$, $\displaystyle\lim_{M\to \infty} \int_{-\pi}^{\pi} D_M(z)g(x+z)dz=g(x)$ uniformly in $x$ with $\displaystyle D_M(z)=\frac 1 {2\pi} \frac {\sin\big ((2M+1)z/2\big)}{\sin\big (z/2\big)}=\frac 1 {2\pi} \sum_{k=-M}^M e^{ikz}$. Therefore with $h$ a ${\cal C}^1$ function $2\pi$-periodic for each component,
$$\lim_{M\to \infty} \int_{[-\pi,\pi]^4}\hspace{-0.7cm} 2\pi \, D_M((\lambda-\lambda')+(\mu-\mu')) h(\lambda,\lambda',\mu,\mu')d\lambda d\lambda' d\mu d\mu'=2\pi \, \int_{[-\pi,\pi]^3}\hspace{-0.7cm} h(\lambda'-\mu+\mu',\lambda',\mu,\mu')d\lambda' d\mu d\mu';
$$
Therefore,
\begin{multline}\label{C1}
C_1 =  2\pi \, (\mu_4-1) \,a \, a'\,  \int_{[-\pi,\pi]^3}\hspace{-0.7cm}d\lambda' d\mu d\mu' e^{i(\mu-\mu')
 (b'-b)} \\
\times S_{a}(a(\lambda'-\mu+\mu'))\widehat \alpha (\lambda'-\mu+\mu')\overline{S_{a}(a\lambda')}\overline{\widehat \alpha (\lambda')}S_{a'}(a'\mu)\widehat \alpha (\mu)\overline{S_{a'}(a'\mu')}\overline{\widehat \alpha (\mu')}.
\end{multline}
* {\it Case 2:} in such a case, with $s_1\neq s_2$, $\Cov\big(\xi_{s_{1}}\xi_{s_{2}},\xi_{s_{1}}\xi_{s_{2}}\big)=1$ and 
\begin{eqnarray*}
\nonumber
C_2\hspace{-3mm} &=& \hspace{-3mm}\frac 2 {a \, a'} \sum_{(s,s')\in \Z^2,s\neq s'} \sum_{t_1=1}^N 
\alpha(t_1-s)\psi(\frac{t_1-b}{a})\sum_{t_2=1}^N \alpha(t_2-s)\psi(\frac{t_2-b'}{a'})   \sum_{t_3=1}^N 
\alpha(t_3-s')\psi(\frac{t_3-b}{a}) \sum_{t_4=1}^N \alpha(t_4-s')\psi(\frac{t_4-b'}{a'})\\
&=& \hspace{-3mm}-\frac {2C_1}{\mu_4-1}+\frac 1 {a \, a'} \sum_{(s,s')\in \Z^2} \sum_{t_1=1}^N 
\alpha(t_1-s)\psi(\frac{t_1-b}{a})\sum_{t_2=1}^N \alpha(t_2-s)\psi(\frac{t_2-b'}{a'})   \sum_{t_3=1}^N 
\alpha(t_3-s')\psi(\frac{t_3-b}{a}) \sum_{t_4=1}^N \alpha(t_4-s')\psi(\frac{t_4-b'}{a'})\\
C_2 \hspace{-3mm} &= & \hspace{-3mm}-\frac {2C_1}{\mu_4-1}+2 \,a \, a'\,  \lim_{M\to \infty}\lim_{M'\to \infty} \int_{[-\pi,\pi]^4}\hspace{-0.7cm}d\lambda d\lambda' d\mu d\mu' e^{i
 [b(\lambda- \mu)-b'(\lambda'-\mu')]} \\
&& \hspace{3cm}\times \sum_{s=-M}^M \sum_{s=-M'}^{M'}e^{is (\lambda'-\lambda)+is'(\mu'-\mu)}S_{a}(a\lambda)\widehat \alpha (\lambda)\overline{S_{a'}(a'\lambda')}\overline{\widehat \alpha (\lambda')}S_{a}(a\mu)\widehat \alpha (\mu)\overline{S_{a'}(a'\mu')}\overline{\widehat \alpha (\mu')}\\
 &= & \hspace{-3mm}-\frac {2C_1}{\mu_4-1}+ 8\pi^2  \,a \, a'\, \int_{[-\pi,\pi]^2}\hspace{-0.7cm} e^{i(\lambda-\mu)
 (b-b')} S_{a}(a\lambda) \overline{S_{a'}(a'\lambda)} S_{a}(a\mu)\overline{S_{a'}(a'\mu)} \, \times \big |\widehat \alpha (\lambda)\big |^2 \, \big |\widehat \alpha (\mu)\big |^2 d\lambda d\mu ,
\end{eqnarray*}
using the asymptotic behaviors of two Dirichlet kernels.\\
Now we have to compute $\displaystyle \sum_{b=1}^{N-a}
\sum_{b'=1}^{N-a'} (C_1+C_2)$. In both cases ($C_1$ and $C_2$), one again obtains a function of a Dirichlet kernel:
\begin{eqnarray}\label{FN}
F_N(a,a',v):=\sum_{b=1}^{N-a} \sum_{b'=1}^{N-a'}e^{i\, v\, (b-b')} =e^{iv(a-a')/2}\frac{\sin((N-a)v/2)\, \sin((N-a')v/2)}{\sin^2(v/2)}.
\end{eqnarray}
For a continuous function $h:[-\pi, \pi] \mapsto \R$, 
$$
 \lim_{N \to \infty} \frac 1 {N} \int_{-\pi}^ \pi h(v) F_N(a,a',v)dv=\lim_{N \to \infty} \frac 1 {N^2} \int_{-\pi N}^ {\pi N} h(\frac v N) F_N(a,a',\frac v N )dv=4h(0)\, \int_{-\infty}^\infty \frac {\sin^2(v/2)}{v^2}dv=2\pi h(0),
$$ thanks to Lebesgue Theorem and with $a/N\to 0$ ($N\to 0$). Then, from (\ref{C1}),
\begin{eqnarray}
\nonumber N \, \frac 1 {N-a} \frac 1 {N-a'}
\sum_{b=1}^{N-a} \sum_{b'=1}^{N-a'} C_1\hspace{-3mm} &\sim & \hspace{-3mm} 4\pi^2 \, (\mu_4-1)  aa'  \int_{[-\pi,\pi]^2}\hspace{-0.7cm}d\lambda' d\mu' |S_{a}(a\lambda')|^2\, |S_{a'}(a'\mu')|^2\,| \widehat \alpha (\lambda')|^2 | \widehat \alpha (\mu')|^2 \\
\nonumber  &\sim  & \hspace{-3mm} 4\pi^2 \, (\mu_4-1)  \int_{-a\pi}^{a\pi} \hspace{-0.2cm}|S_{a}(\lambda)|^2\,| \widehat \alpha (\lambda/a)|^2 d\lambda  \int_{-a'\pi}^{a'\pi} \hspace{-0.2cm}|S_{a}(\mu)|^2\,| \widehat \alpha (\mu/a')|^2 d\mu \\
\label{casC12}  && \hspace{-3cm} \Longrightarrow N \,\frac {(aa')^{-2d}(c_dK_{(\psi,2d)})^{-2} } { 4\pi^2(N-a)(N-a')}
\sum_{b=1}^{N-a} \sum_{b'=1}^{N-a'} C_1  \limiteN  (\mu_4-1) \\
\label{casC1}  && \hspace{-3cm} \mbox{and}\qquad \frac N {a_N}\,\frac {(ra_Nr'a_N)^{-2d}(c_dK_{(\psi,2d)})^{-2} } { 4\pi^2(N-ra_N)(N-r'a_N)}
\sum_{b=1}^{N-ra_N} \sum_{b'=1}^{N-r'a_N} C_1  \limiteN  0 ,
\end{eqnarray}
using the same arguments than in Property \ref{cor1} since $a_N\to \infty$ (and therefore $a\to \infty$ and $a'\to \infty$). \\
Moreover, if we consider that $a=ra_N$ and $a'=r'a_N$,
\begin{eqnarray*}
N \, \frac 1 {N-a} \frac 1 {N-a'} \sum_{b=1}^{N-a} \sum_{b'=1}^{N-a'} C_2\hspace{-3mm} &\sim & \hspace{-3mm} 16\pi^3 aa'   \int_{-\pi}^\pi \hspace{-2mm}    \big |S_{a}(a\lambda) \big |^2  \big | S_{a'}(a'\lambda)\big |^2  \big |\widehat \alpha (\lambda)\big |^4d\lambda    -\frac {2N}{\mu_4-1} \frac 1 {N-a} \frac 1 {N-a'} \sum_{b=1}^{N-a} \sum_{b'=1}^{N-a'} C_1 \\
 && \hspace{-38mm} \sim16\pi^3 rr'a_N   \int_{-a_N\pi}^{a_N\pi} \hspace{-2mm}    \big |S_{ra_N}(r\lambda) \big |^2  \big | S_{r'a_N}(r'\lambda)\big |^2  \big |\widehat \alpha (\lambda/a_N)\big |^4 d\lambda    -\frac {2N}{\mu_4-1} \frac 1 {N-ra_N} \frac 1 {N-r'a_N} \sum_{b=1}^{N-ra_N} \sum_{b'=1}^{N-r'a_N} C_1 \\ 
&& \hspace{-3.9cm} \Longrightarrow \frac N {a_N}  \,  \frac {(r\, r'\, a_N^2)^{-2d}(c_dK_{(\psi,2d)})^{-2} } {4\pi^2(N-ra_N)(N-r'a_N)}
\sum_{b=1}^{N-ra_N} \sum_{b'=1}^{N-r'a_N} C_2  \limiteN    4\pi \frac {(rr')^{1-2d}}{K^2_{(\psi,2d)} }\int_{-\infty}^{\infty} \hspace{-2mm}   \frac{ \big |\widehat \psi (r\lambda) \big |^2  |\widehat \psi (r'\lambda) \big |^2}{\lambda^{4d}} d\lambda, 
 \end{eqnarray*}
always using the same trick than in Property \ref{cor1} since $a\to \infty$ and $a'\to \infty$.
Therefore, with (\ref{casC1}), one deduces that:
\begin{eqnarray}\label{limitcov}
\frac N {a_N} \, \Cov (\widetilde T_N(r\,a_N), \widetilde
T_N(r'\,a_N)) \limiteN 4\pi \frac {(rr')^{1-2d}}{K^2_{(\psi,2d)} }\int_{-\infty}^{\infty} \hspace{-2mm}   \frac{ \big |\widehat \psi (r\lambda) \big |^2  |\widehat \psi (r'\lambda) \big |^2}{\lambda^{4d}} d\lambda.
\end{eqnarray}
Note that if $r=r'$ then $\displaystyle \frac N {r\, a_N} \, \Var (\widetilde T_N(r\,a_N)) \limiteN \sigma_\psi^2(d)=64\pi^5\, \frac {K_{(\psi *\psi,4d)}}{K^2_{(\psi,2d)} }$ only depending on $\psi$ and $d$. \\
~\\
{\bf Step 2} We prove here that if the distribution of the innovations $(\xi_t)_t$ is such that there exists $r>0$ satisfying $\E \big (e^{r\xi_0}\big )\leq \infty$ (condition so-called the Cramèr condition), then for any $a\in \N^*$, $ (\widetilde T_N(r_i\,a_N))_{1\leq i \leq \ell}=\Big (\frac 1 {N-r_ia_N} \, \sum_{k=1}^{N-r_ia_N} \widetilde e^2(r_ia_N,k)\Big )_{1\leq i \leq \ell}$ satisfies a central limit theorem. Such theorem is implied by the proof that $\sqrt {\frac N {a_N}} \,   \sum _{i=1}^\ell \frac {u_i} {N-r_ia_N} \, \sum_{k=1}^{N-r_ia_N} \widetilde e^2(r_ia_N,k)$ asymptotically follows a Gaussian distribution for any  $(u_i)_{1\leq i \leq \ell}\in \R^\ell$. For establishing this result we are going to adapt a proof of Giraitis (1985) which shows central limit theorems for function of linear processes using a decomposition with Appell polynomials. Indeed we specified that $X$ is a two-sided linear process and therefore martingale type results as in Wu (2002) or Furmanczyk (2007) can not be applied. Moreover, since $a_N$ is a sequence depending on $N$ it is required to prove a central limit theorem for triangular arrays. Unfortunately the recent paper of Roueff and Taqqu (2009) dealing with central limit theorems for arrays of decimated linear processes, and which can be applied to establish a multidimensional central limit for the variogram of wavelet coefficients associated to a multi-resolution analysis can not be applied here because in this paper this variogram is defined as in (\ref{samplevarusual}) with coefficients taken every $n/n_j$ ($\simeq a_N$ with our notation) and the mean of $n_j$ ($N/a_N$ with our notation) coefficients is considered (and the convergence rate is $\sqrt{n_j}$). Our definition of the wavelet coefficient variogram  (\ref{samplevar}) is an average of $N-a_N$ terms and the convergence rate is $N/a_N$. Then we chose to adapt the results and method of Giraitis (1985). \\
More precisely, consider the case $\ell=1$. For $a>0$, $(\widetilde e(a,b))_{1\leq b \leq N-a}$ is a stationary linear process satisfying assumptions of the paper of Giraitis (called $X_t$ in this article). Now we consider $H_2(x)=x^2-1$ the second-order Hermite polynomial and would like to prove that 
$$ 
\Big ( \frac N {a_N} \Big )^{1/2} \,\frac 1 {N-a_N}  \sum_{b=1}^{N-a_N}\big ( \widetilde e^2(a_N,b)-1 \big ) \simeq  \Big ( \frac 1  {Na_N} \Big )^{-1/2} \, \sum_{b=1}^{N-a_N} H_2(\widetilde e(a_N,b)) \limiteloiN {\cal N}\big (0,\sigma_\psi^2(d) \big ).
$$
Now since the distribution of $\xi_0$ is supposed to satisfy the Cramèr condition, following the proof of Proposition 6 (Giraitis, 1985), define $S_N^{(n)}=\sum_{b=1}^{N-a_N} A_n^{(a_N)}(\widetilde e(a_N,b))$ where $A_n^{(a_N)}$ is the Appell polynomial of degree $n$ corresponding to the probability distribution of $\widetilde e(a_N,\cdot)$. We are going to prove that the cumulants of order $k\geq 3$ are such as
\begin{equation}\label{cum}
\chi \big (S_N^{(n(1))},\ldots, S_N^{(n(k))}\big) =o\big ( ( N a_N ) ^{k/2} \big)
\end{equation}  
for any $n(1),\cdots, n(k)\geq 2$ (the computation of the cumulant of order $2$ is induced by Step 1 of this proof). Indeed, $  \chi \big (S_N^{(n(1))},\ldots, S_N^{(n(k))}\big)=\sum_{\gamma \in \Gamma_0(T)} d_\gamma I_\gamma(N)$ where $\Gamma_0(T)$ is the set of possible diagrams and the definition of $I_\gamma(N)$ is provided in (34) of Giraitis (1985). \\
In the case of Gaussian diagrams, $I_\gamma(N)=o\big ((Na_N) ^{k/2} \big)$, since this case is induced by the Gaussian case and the second order moments. 
\\
If $\gamma$ is a non Gaussian diagram, {\em mutatis mutandis}, we are going to follow the notation and proof of Lemma 2 of Giraitis (1985). Note first 
from Step 1, we can write:
\begin{eqnarray}\label{tilde}
\widetilde e(a,b)=\sum_{s\in \Z}  \, \beta_a(b-s)\,\xi_s\quad \mbox{with}\quad \beta_a(s)=\frac {\sqrt a}{\sqrt{\E e^2(a,b)}} \, \int_{-\pi}^\pi S_a(a\lambda) \widehat \alpha(\lambda) e^{i\lambda s} d\lambda.
\end{eqnarray}
Then for $u\in [-\pi,\pi]$,
\begin{eqnarray*}
\widehat \beta_a(u)&=&\frac 1 {2\pi} \sum_{s=-\infty}^\infty \beta_a(s) e^{-isu}\\
&=&\frac {\sqrt a}{2\pi \sqrt{\E e^2(a,b)}}\lim_{m\to \infty} \int_{-\pi}^\pi   \sum_{s=-m}^m S_a(a\lambda) \widehat \alpha(\lambda) e^{is(\lambda-u)} d\lambda\\
&=&\frac {\sqrt a}{ \sqrt{\E e^2(a,b)}} \, S_a(au) \widehat \alpha(u), 
\end{eqnarray*}
with the asymptotic behavior of Dirichlet kernel. Now, in case a/ of Lemma 2 of Giraitis (1985), consider the diagram $V_1=\{(1,1),(2,1),(3,1)\}$ and assume that for the rows $L_j$ of the array $T$,  $j=1,\cdots,k$ ($k\geq 3$),  $|V_1 \cap L_j|\geq 1$ for at least $3$ different rows $L_j$. Then the inequality (39) can be repeated, and on the hyperplane $x_{V_1}$, a part of the integral (34) provides
$$\Big | \int _{\{x_{11}+x_{21}+x_{31}=0\}\cap [-\pi,\pi]^3} \hspace{-2cm} dx_{11}dx_{21}dx_{31}\, \prod _{j=1}^3 D_N((x_{j1}+\cdots+x_{jn(j)})\widehat \beta_a(x_{j1}) \Big |  \leq C \, \alpha_1(u_1)\alpha_2(u_2)\alpha_3(u_3),$$
with $u_i=x_{i2}+\cdots+x_{in(i)}$ and the same expressions of $\alpha_i$ provided in Giraitis (1985).  
It remains to bound $\alpha_i(u)$. But, with the same approximations as in the proof of Property \ref{cor1}, for $a$ and $N$ large enough
\begin{eqnarray*}
\alpha_1(u)&=&\int_{-\pi}^\pi \big |\widehat \beta_a(u)\, D_N(x+u)\big |dx \sim \sqrt{2\pi}\, \frac 1 {\sqrt{a}}\, \int_{-a\pi}^{a\pi}\big |  \frac {\widehat \psi(x)}{|x|^d}\big | \, \big |D_N\big (\frac x a +u\big )\big |du \\
& \leq & 2 \sqrt{a} \sup_{x\in \R} \Big \{  \frac {|\widehat \psi(x)|}{|x|^d} \Big \} \, \int_{-\pi}^\pi |D_N(x+u)|dx \\
 &\leq & 2 C  \sup_{x\in \R}  \Big \{ \frac {|\widehat \psi(x)|}{|x|^d} \Big \}\,  \sqrt{a} \, \log N,
\end{eqnarray*}
since there exists $C>0$ such as $\int_{-\pi}^\pi |D_N(x+u)|dx \leq C \log N$ for any $u \in [-\pi,\pi]$. Now for $i=2,3$, $a$ and $N$ large enough,
\begin{eqnarray*}
\alpha^2_i(u)&= & \|\widehat \beta_{a_N}(\cdot )\, D_N(u+\cdot) \|^2_2 \\
&\leq & 2 \int_{-a_N\pi}^{a_N\pi}  \frac {|\widehat \psi(x)|^2 }{|x|^{2d}} D^2_N\big (\frac x {a_N} +u\big )du \\
&\leq & 2 C  \sup_{x\in \R}  \Big \{\frac {|\widehat \psi(x)|^2}{|x|^{2d}}\Big \} \, a_N  \, \int_{-\pi}^\pi |D^2_N(x+u)|dx \\
& \leq &  C'  \sup_{x\in \R}  \Big \{\frac {|\widehat \psi(x)|^2}{|x|^{2d}}\Big \}\, N a_N.
\end{eqnarray*}
Then $ \alpha_1(u_1)\alpha_2(u_2)\alpha_3(u_3) = o( (Na_N)^{3/2})$. \\
For the $k-3$ other terms, a result corresponding to Lemma 1 of Giraitis (1985) can also be obtained. Indeed, for $a$ and $N$ large enough,
\begin{eqnarray*}
\| g_{N,j}\|_2^2 & =& \int_{[-\pi,\pi]^{n(j)}} dx \,D^2_N(x_1+\cdots+ x_{n(j)} ) \, \prod_{i=1}^{n(j)} |\widehat \beta_a(x_i)|^2 \\
&  \leq &C \,  \int_{[-a\pi,a\pi]^{n(j)}} dx \,D^2_N(\frac 1 a (x_1+\cdots+ x_{n(j))} \prod_{i=1}^{n(j)} \frac {|\widehat \psi(x_i)|^2 }{|x_i|^{2d}}  \\
&\leq & C \, \big |\sup_{x\in \R}  \big \{\frac {|\widehat \psi(x)|^2}{|x|^{2d}}\big \}\big |^{n(j)} \, a \,  \| D_N\big (\cdot \big )\|^2_2 \\
& \leq &  C' \,N a_N
\end{eqnarray*}
with $C'\geq 0$ not depending on $N$ and $a_N$. Thus $\| g_{N,j}\|_2 \leq C\, (Na_N)^{1/2}$ with $C\geq 0$. Using the same reasoning, there also exists $C'\geq 0$ such as $\| g'_{N,j}\|_2 \leq C\, (Na_N)^{1/2}$ for $j\geq 2$ while $\| g'_{N,1}\|_2=O(\sqrt {a_N}\,\log N)=o((Na_N)^{1/2})$. As a consequence, for $\gamma$ such as $|V_1 \cap L_j|\geq 1$ for at least $3$ different rows $L_j$, and more generally with $|V_1|\geq 3$,
\begin{equation} \label{majoIN}
I_\gamma(N)=o\big ((Na_N) ^{k/2} \big).
\end{equation}
For other $\gamma$, it remains to bound the function $h(u_1,u_2)$ defined in Giraitis (1985, p. 32) as follows (with $x=x_{11}+x_{12}$) and with $u_1+u_2\neq 0$:
\begin{eqnarray*}
h(u_1,u_2) & = & \Big ( \int_{-\pi}^{\pi} \big | \widehat \beta_{a_N}(-x )\, D_N(u_1+x)\, D_N(u_2-x) \big |dx  \Big )\, \Big (\int _{-\pi}^{\pi} \big | \widehat \beta_{a_N}(x )\big |^2dx \Big ) \\
& \leq & \big |\sup_{x\in \R}  \big \{\frac {|\widehat \psi(x)|^2}{|x|^{2d}}\big \}\big |  \, a_N \,   \Big ( \int_{-\pi}^{\pi}  \big |D_N\big (u_1+ x  \big )\, D_N\big (u_2-x \big ) \big |dx  \Big ) \Big (2\pi \, \int _{-\infty}^{\infty} \frac {|\widehat \psi(x)|^2}{|x|^{2d}} \,dx \Big ) .
\end{eqnarray*} 
But 
\begin{eqnarray*}
 \int_{-\pi}^{\pi}  \big |D_N\big (u_1+ x  \big )\, D_N\big (u_2-x \big ) \big |dx & \leq & 2 \int_{-2\pi N}^{2\pi N} \Big | \frac {\sin(x) } x \, \frac {\sin (\frac N 2 (u_1+u_2)-x)}{\sin(\frac 1 2 (u_1+u_2)-\frac x N)} \Big |dx \\
& \leq & \left \{ \begin{array}{ll} C \, \log N \big | \sin(\frac 1 2 (u_1+u_2)) \big |^{-1}  & \mbox{if $|u_1+u_2| \geq (N\log N)^{-1}$} \\ 
 C \, N  & \mbox{if $|u_1+u_2|<(N\log N)^{-1}$}
 \end{array} \right .  .
\end{eqnarray*} 
Therefore,
\begin{eqnarray*}
\| h(u_1,u_2)\|_2^2=\int_{[-\pi,\pi]^2} h^2(u_1,u_2)du_1du_2 & \leq  &  C \,a_N^2 \, \Big ( \log^2 N \int_{(N\log N)^{-1}}^\pi (\sin x)^{-2}\,dx + N^2 \, \int_0^{(N\log N)^{-1}} dx \Big ) \\
& \leq & C \,a_N^2 \,\big ( N \log^3 N + N \log N \big ),
\end{eqnarray*} 
and hence $\| h(u_1,u_2)\|_2=o(N a_N)$. Finally, (\ref{majoIN}) holds for all $\gamma$ and it implies (\ref{cum}). \\
If $\ell>1$, the same proof can be repeated from the linearity properties of cumulants. Thus, $ (\widetilde T_N(r_i\,a_N))_{1\leq i \leq \ell}$ satisfies the following central limit:
\begin{eqnarray}\label{tlcnorm}
\sqrt{ \frac N {a_N}} \, \big ( \widetilde T_N(r_i\,a_N)-1\big )_{1\leq i \leq \ell} \tend {\cal N} \big ( 0 \, , \, \Gamma(r_1,\cdots,r_\ell,\psi,d)\big ),
\end{eqnarray}
with $\Gamma(r_1,\cdots,r_\ell,\psi,d)=(\gamma_{ij})_{1\leq i,j\leq \ell}$ given in (\ref{cov}). \\
~\\
{\bf Step 3} Now we extend the central limit obtained in Step 2 for linear processes with an innovation distribution satisfying a Cramèr condition ($\E \big (e^{r\xi_0}\big )<\infty$) to the weaker condition $\E \xi^4_0<\infty$ using a truncation procedure. Thus assume now that $\E \xi_0^4<\infty$. Let $M>0$ and define $\xi_t^{-}=\xi_t \, \I_{|\xi|\leq M}$ and $\xi_t^+=\xi_t \, \I_{|\xi|> M}$, $\widetilde e^-(a,b)=\sum_{s\in \Z}  \, \beta_a(b-s)\,\xi_s^- $ and $\widetilde e^+(a,b)=\sum_{s\in \Z}  \, \beta_a(b-s)\,\xi_s^+$ using (\ref{tilde}). Clearly $\widetilde e(a,b)=\widetilde e^+(a,b)+\widetilde e^-(a,b)$. We are going to prove that (\ref{tlcnorm}) holds. For this, we begin by writing
\begin{eqnarray}\label{decompo}
 \widetilde T_N(r_i\,a_N)-1&=& \frac 1 {N-r_ia_N}\Big ( \sum_{b=1}^{N-r_ia_N}\big (\widetilde e^-(r_ia_N,b)\big )^2 -1\big)-2 \widetilde e^+(r_ia_N,b) \widetilde e^-(r_ia_N,b)+\big (\widetilde e^+(r_ia_N,b)\big )^2\Big )~~~~ ~~~~
\end{eqnarray}
We first prove that $\big (\widetilde T_N^-(r_i\,a_N)-1\big)_{1\leq i \leq \ell}=\big (\frac 1 {N-r_ia_N}\, \sum_{b=1}^{N-r_ia_N}\big (\widetilde e^-(r_ia_N,b)\big )^2 -1\big)_{1\leq i \leq \ell}$ also satisfies (\ref{tlcnorm}). Indeed, $(\widetilde e^-(r_ia_N,b))$ is a linear process with innovations $(\xi_t^{-})$ satisfying the Cramèr condition and it is obvious that $\Big (\frac {\E \big (\widetilde e(r_ia_N,b)\big )^2}{\E \big (\widetilde e^-(r_ia_N,b)\big )^2}\Big )^{1/2}\, \widetilde e^-(r_ia_N,b)_{b,i}$ has exactly the same distribution than $\widetilde e(r_ia_N,b)_{b,i}$. Therefore it remains to prove that $\sqrt {\frac{N}{a_N}}\, \Big (\frac {\E \big (\widetilde e(r_ia_N,b)\big )^2}{\E \big (\widetilde e^-(r_ia_N,b))^2}-1 \Big )$ converges to $0$. We have $\E \big (\widetilde e(r_ia_N,b))^2=\big (\sum_{s\in \Z}  \, \beta^2_a(s)\big )\,\E (\xi_0)^2=1$ and $\E \xi_0^2=1$ (from Property \ref{cor1}). Then
\begin{eqnarray*}
\Big |\frac {\E \big (\widetilde e^-(r_ia_N,b))^2}{\E \big (\widetilde e(r_ia_N,b))^2}-1\Big |&\leq & 2 \, \Big ( \E \big (\widetilde e^+(r_ia_N,b))^2 \Big ) ^{1/2} + \E \big (\widetilde e^+(r_ia_N,b))^2.
\end{eqnarray*}
We have $\E \big (\widetilde e^+(r_ia_N,b))^2=\big (\sum_{s\in \Z}  \, \beta^2_a(s)\big )\,\E (\xi^+_0)^2=\E (\xi^+_0)^2$ from previous arguments and since we assume that the distribution of $\xi_0$ is symmetric. But using Hölder's and Markov's inequalities $\E (\xi^+_0)^2\leq (\E \xi_0^4)^{1/2} (\Pr(|\xi_0|>M))^{1/2} \leq (\E \xi_0^4)\, M^{-2}$. Hence, there exists $C>0$ not depending on $M$ and $N$,
$$
\sqrt {\frac{N}{a_N}}\, \Big |\frac {\E \big (\widetilde e^-(r_ia_N,b))^2}{\E \big (\widetilde e(r_ia_N,b))^2}-1\Big | \leq \frac C M \, \sqrt {N}{a_N} \limiteN 0
$$
when $M=N$ (for instance). Therefore $\big (\widetilde T_N^-(r_i\,a_N)-1\big)_{1\leq i \leq \ell}$ satisfies the CLT (\ref{tlcnorm}). \\
From (\ref{decompo}), it remains to prove that 
$$
\sqrt {\frac{N}{a_N}}\,  \frac 1 {N-r_ia_N}\Big ( \sum_{b=1}^{N-r_ia_N}-2 \widetilde e^+(r_ia_N,b) \widetilde e^-(r_ia_N,b)+\big (\widetilde e^+(r_ia_N,b)\big )^2\Big )\limiteprobaN 0.
$$
From Markov's and Hölder inequalities, this is implied when $\sqrt {\frac{N}{a_N}}\,  \big ( \E \big (\widetilde e^+(r_ia_N,b)\big )^2+2 \sqrt {\E \big (\widetilde e^+(r_ia_N,b)\big )^2} \big ) \limiteN 0$ with $\E \big (\widetilde e^+(r_ia_N,b)\big )^2=1$. Using $\E \big (\widetilde e^+(r_ia_N,b))^2\leq (\E \xi_0^4)\, M^{-2}$ obtained above, we deduce that this statement holds when $M=N$ (for instance). As a consequence, from (\ref{decompo}), the CLT (\ref{tlcnorm}) holds even if the distribution of $\xi_0$ is only symmetric and such that $\E \xi_0^4<\infty$.\\
~\\
{\bf Step 4} It remains to apply the Delta-method to (\ref{tlcnorm}) with function $(x_1,\cdots,x_\ell)\mapsto (\log x_1,\cdots,\log x_\ell)$:
\begin{eqnarray*}
\sqrt{ \frac N {a_N}} \, \big ( \log \big ( T_N(r_i\,a_N)\big )-\log (\E e^2(a_N,1)) \big )_{1\leq i \leq \ell} \tend {\cal N} \big ( 0 \, , \, \Gamma(r_1,\cdots,r_\ell,\psi,d)\big ),
\end{eqnarray*}
With $\E e^2(a_N,1)$ provided in Property \ref{cor1}, we obtain 
$$
\displaystyle \log \E e^2(a_N,1)=2d \, \log (a_N)+\log \big (\frac { c_d K_{(\psi,2d)}} {2\pi}  \big ) +\frac {c_{d'}K_{(\psi,2d-d')}} {2\pi \, a_N^{d'}} \big (1+o(1) \big )
$$
Therefore, when $\displaystyle \sqrt{\frac {N}{a_N}} \frac 1 {a_N^{d'}} \limitN 0$, {\it i.e.} $N^{\frac 1 {1+2d'}}=o(a_N)$, the CLT (\ref{CLTSN}) holds.  \end{proof}
\begin{proof}[Proof of Theorem \ref{tildeD}] Here we use Theorem 1 of Bardet {\it et al.} (2008) where it was proved that the CLT (\ref{CLTSN}) is still valid when $a_N$ is replaced by $N^{\widetilde \alpha_N}$. Then, since $\widetilde d_N=\widetilde M_N\,  Y_N(\widetilde \alpha_N)$ with $\widetilde M_N=\big (0~1/2 \big )\big (Z_{1}'\widehat \Gamma_N^{-1} Z_{1}\big )^{-1} Z_{1}'\widehat \Gamma_N^{-1}$ we deduce that $\sqrt{N/N^{\widetilde \alpha_N}} \big (\widetilde d_N-d\big)$ is asymptotically Gaussian with asymptotic variance the limit in probability of $\widetilde M_N\,\Gamma(1,\ldots,\ell,d,\psi)\, \widetilde M_N'$, that is $\sigma^2$.\\
The right hand side relation of (\ref{CLTD2}) is also an obvious consequence of Theorem 1 of Bardet {\it et al.} (2008). 
\end{proof}
\begin{proof}[Proof of Theorem \ref{tildeT}] 
The theory of linear models can be applied: $Z_{N^{\widetilde \alpha_N}} \Big( \begin{array}{c} \widetilde  c_N\\ 2\widetilde  d_N\end{array}\Big )$ is an orthogonal projector of $Y_N(\widetilde \alpha_N)$ on a subspace of dimension $2$, therefore $Y_N(\widetilde \alpha_N)-Z_{N^{\widetilde \alpha_N}} \Big( \begin{array}{c} \widetilde  c_N\\ 2\widetilde  d_N\end{array}\Big )$ is an orthogonal projector of $Y_N(\widetilde \alpha_N)$ on a subspace of dimension $\ell-2$. Moreover, using the CLT (\ref{CLTSN}) where $a_N$ is replaced by $N^{\widetilde \alpha_N}$, we deduce that $\sqrt{N/N^{\widetilde \alpha_N}} \widehat \Gamma_N^{-1} Y_N(\widetilde \alpha_N)$ asymptotically follows a Gaussian distribution with asymptotic covariance matrix $I_\ell$ (identity matrix). Hence form Cochran Theorem we deduce (\ref{Testconv}). 
\end{proof}\bibliographystyle{amsalpha}

\end{document}